\documentclass{amsart}
\usepackage[english]{babel}
\usepackage[latin1]{inputenc}
 \usepackage[all]{xy}
 \usepackage[pagebackref]{hyperref}

\usepackage{amsmath,amsfonts,amssymb,amsthm,amscd,array,stmaryrd,mathrsfs, mathdots}
\usepackage[makeroom]{cancel}
\PassOptionsToPackage{option}{xcolor}
\usepackage{pstricks}
\usepackage{tikz}
\usepackage{graphicx}

\theoremstyle{plain}
\newtheorem{thm}{Theorem}

\newtheorem{lem}{Lemma}[section]
\newtheorem{cor}[lem]{Corollary}
\newtheorem{prop}[lem]{Proposition}

\theoremstyle{definition}

\newtheorem{rem}[lem]{Remark}
\newtheorem{ex}[lem]{Example}

\let\ssection=\section
\renewcommand{\section}{\setcounter{equation}{0}\ssection}

\newcommand{\Z}{\mathbb{Z}}

\newcommand{\T}{\mathbb{T}}
\newcommand{\Q}{\mathbb{Q}}

\newcommand{\G}{\mathcal{G}}

\newcommand{\Rc}{\mathcal{R}}
\newcommand{\Sc}{\mathcal{S}}

\newcommand{\coar}{\textup{coar}}
\renewcommand{\ar}{\textup{ar}}
\newcommand{\wt}{\textup{wt}}

\newcommand{\Tr}{\textup{Tr\,}}

\newcommand{\GL}{\mathrm{GL}}

\newcommand{\SL}{\mathrm{SL}}

\begin{document}

\title{Quantum continuants, quantum rotundus and triangulations of annuli}

\author[L. Leclere, S. Morier-Genoud]{Ludivine Leclere and Sophie Morier-Genoud}

\address{Sophie Morier-Genoud, 
Universit\'e de Reims Champagne Ardenne, CNRS, LMR, Reims, France
}
\email{sophie.morier-genoud@univ-reims.fr}

\address{
Ludivine Leclere, Universit\'e de Reims Champagne Ardenne, CNRS, LMR, Reims, France
}
\email{ludivine.leclere@univ-reims.fr}
\keywords{$q$-analogues, continued fractions, modular group, palindromic polynomials, unimodality, continuants, rotundus, triangulations, friezes}

\maketitle

\begin{abstract}
We give enumerative interpretations of the polynomials arising as numerators and denominators of the $q$-deformed rational numbers introduced by Morier-Genoud and Ovsienko. The considered polynomials are quantum analogues of the classical continuants and of their cyclically invariant versions called rotundi. The combinatorial models involve triangulations of polygons and annuli.
We prove that the quantum continuants are the coarea-generating functions of paths in a triangulated polygon and that the quantum rotundi are the (co)area-generating functions of closed loops on a triangulated annulus.
\end{abstract}

\section{Introduction}
Continuants are determinants of tridiagonal matrices. 
They have a long history going back to Euler's work on continued fractions \cite[Chap. 18]{Eul}, see also \cite{Muir}, \cite{Per}, \cite{concrete}.
The name, coming from the fusion of ``continued fraction'' and ``determinant'', was introduced by Thomas Muir in the middle of the 19th century \cite{Muir2}.

The following two types of continuants:
\begin{equation}
\label{ContEq}
K_{n}(a_1,\ldots,a_n):=
\left|
\begin{array}{cccccc}
a_1&1&&&\\[4pt]
-1&a_{2}&1&&\\[4pt]
&\ddots&\ddots&\!\!\ddots&\\[4pt]
&&-1&a_{n-1}&\!\!\!\!\!1\\[4pt]
&&&\!\!\!\!\!-1&\!\!\!\!a_{n}
\end{array}
\right| 
\text{ and }
E_{k}(c_1,\ldots,c_k)
:=
\left|
\begin{array}{cccccc}
c_1&1&&&\\[4pt]
1&c_{2}&1&&\\[4pt]
&\ddots&\ddots&\!\!\ddots&\\[4pt]
&&1&c_{k-1}&\!\!\!\!\!1\\[4pt]
&&&\!\!\!\!\!1&\!\!\!\!c_{k}
\end{array}
\right|
\end{equation}
 are polynomial expressions in the variables $a_i$ and $c_i$ respectively, with the following conventions for empty sets of variables:
$K^{}_{0}()_{}=E_{0}()=1$ and $K^{}_{-1}()_{}=E_{-1}()=0$.

Continuants are related to the numerators and denominators of the regular continued fractions
\begin{eqnarray}\label{CFsEq1}
a_1 + \cfrac{1}{a_2 
         +\cfrac{1}{\ddots + \cfrac{1}{a_n} } } 
          \quad
          &=&\quad
\dfrac{K_{n}(a_1,\ldots,a_n)}{K_{n-1}(a_2,\ldots,a_n)}
\end{eqnarray}
and of the negative signed continued fractions
\begin{eqnarray}\label{CFsEq2}
c_1 - \cfrac{1}{c_2 
          - \cfrac{1}{\ddots - \cfrac{1}{c_k} } } 
          \quad&=&\quad
\dfrac{E_{k}(c_1,\ldots,c_k)}{E_{k-1}(c_2,\ldots,c_{k})},
\end{eqnarray}
also known as Hirzebruch-Jung continued fractions.

The continuants also appear as the entries of $2\times 2$-matrices computed by multiplication of elementary matrices as follows
\begin{equation}
\label{MatEq}
\begin{array}{rlcl}
M^+(a_1,\ldots, a_n):=&
\left(
\begin{array}{cc}
a_1&1\\[4pt]
1&0
\end{array}
\right)
\cdots
\left(
\begin{array}{cc}
a_n&1\\[4pt]
1&0
\end{array}
\right)
&=&\left(
\begin{array}{cc}
K_{n}(a_1,\ldots,a_n)&K_{n-1}(a_1,\ldots,a_{n-1})\\[4pt]
K_{n-1}(a_2,\ldots,a_n)&K_{n-2}(a_2,\ldots,a_{n-1})
\end{array}
\right)\\[18pt]
M(c_1,\ldots, c_k):=&
\left(
\begin{array}{cc}
c_1&-1\\[4pt]
1&0
\end{array}
\right)
\cdots
\left(
\begin{array}{cc}
c_k&-1\\[4pt]
1&0
\end{array}
\right)
&=&\left(
\begin{array}{cc}
E_{k}(c_1,\ldots,c_k)&-E_{k-1}(c_1,\ldots,c_{k-1})\\[4pt]
E_{k-1}(c_2,\ldots,c_{k})&-E_{k-2}(c_2,\ldots,c_{k-1})
\end{array}
\right)
\end{array}
\end{equation}

The following combinations of continuants
\begin{equation}
\label{RotsEq}
\begin{array}{lclcl}
R^+ (a_1,\ldots,a_n)&:=&K_n (a_1,\ldots,a_n)+
K_{n-2} (a_2,\ldots,a_{n-1})&=& \Tr M^+(a_1,\ldots, a_n)
\\[12pt]
R (c_1,\ldots,c_k)&:=&E_{k}(c_1,\ldots,c_k)-E_{k-2}(c_2,\ldots,c_{k-1})
&=& \Tr M(c_1,\ldots, c_k)
\end{array}
\end{equation}
are called {\it rotundi}.
The  rotundus $R (c_1,\ldots,c_k)$ was introduced and studied in \cite{CoOv2}.
 Since rotundi are the traces of the matrices \eqref{MatEq}  they are invariant under cyclic permutations on the tuples of variables $(a_1,\ldots, a_n)$ or $(c_1, \ldots, c_k)$.

For the classical continuants $K_n (a_1,\ldots,a_n)$ several enumerative interpretations are known, e.g. in terms of perfect matchings in snake graphs  \cite{CaSc} or paths in lotuses \cite{PoP}. The continuants $E_{k}(c_1,\ldots,c_k)$ are entries in Coxeter's friezes \cite{Cox} and enumerative interpretations are known in terms of triangulations of polygons \cite{CoCo}, \cite{BCI}, or snake graphs \cite{Pro}, see also \cite{MGblms}. In terms of friezes the rotundus $R (c_1,\ldots,c_k)$ corresponds to ``growth coefficients'' that were studied in \cite{BFPT}, \cite{GMV2}.

In the present paper we study $q$-analogues (also called ``quantum analogues'' or ``$q$-deformations'') of continuants and rotundi and their enumerative interpretations. The $q$-analogues we consider come from the theory of $q$-deformations of rational numbers and of continued fractions initiated in \cite{MGOfmsigma} and \cite{MGOexp}. The combinatorial models involve triangulations of polygons and triangulations of annuli.

In Section \ref{qFar} we review on the notion of $q$-rationals and give enumerative interpretations for the numerators and denominators (which are continuants) using oriented paths in the Farey tessellation.

In Section \ref{contsec} we define the $q$-analogues of the objects introduced in this introduction. We reformulate the results of the previous section in terms of continuants and triangulations of polygons.

In Section \ref{rotund} we prove our main result giving an enumerative interpretation of the $q$-rotundus involving closed loops in a triangulated annulus.

Finally, in Section \ref{fin} we collect some extra facts and observations about $q$-rotundi. In particular we discuss links with matchings, dual graphs, Pfaffians, Euler-Minding algorithm.
\section{$q$-analogues of rationals and Farey tessellation}\label{qFar}

The classical $q$-analogues of integers are the following polynomials in $q$ or $q^{-1}$
\begin{equation}\label{qint}
\begin{array}{lcl}
[n]_{q}&=&\frac{1-q^n}{1-q}=1+q+q^{2}+\cdots+q^{n-1}\;,\\[10pt]
[-n]_{q}&=&\frac{1-q^{-n}}{1-q}=-q^{-1}-q^{-2}-\cdots-q^{-n}\;,
\end{array}
\end{equation}
where $n$ is a positive integer.
We also assume $[0]_{q}=0$.

In \cite{MGOfmsigma} $q$-analogues of rational numbers were introduced, extending the above notion of $q$-integers. The approach is based on combinatorial properties of the rational numbers related to the Farey tessellation and to the continued fraction expansions. The subject has led to further developments in various directions.  Notably there are established  links with knots invariants \cite{LeSc2}, \cite{KoWa}, the modular group and the Picard group \cite{LMGadv}, \cite{OvsC}, combinatorics of posets \cite{McCSS}, \cite{Og}, \cite{OgRa}, Markov numbers and Markov-Hurwitz approximation theory \cite{GiVa},  \cite{Kog}, \cite{LaLa}, \cite{LMGOV}, geometry of Grassmannians \cite{Ove}, triangulated categories \cite{BBL}.

\subsection{Recursive definition of $q$-rationals with Farey tessellation}\label{farat}
In this section, following \cite{MGOfmsigma} we define the $q$-analogues of positive rational numbers
using the Farey tessellation. Details on the Farey tessellation can be found in e.g. \cite{HaWr}.

The $q$-analogues are defined for arbitrary rationals but for our combinatorial purpose we will restrict ourself to positive rationals.
We always assume a rational to be written in the irreducible form, i.e. with coprime positive numerators and denominators. Moreover we add an infinity point represented by the ratio~$\frac10$.

The elements of $\Q_{>0}\cup\{\frac10\}$ are ordered on a horizontal segment drawn in the plane with endpoints $\frac01$ at the left and $\frac10$ at the right. 
The Farey tessellation consists of a collection of triangles whose vertices are the rational numbers and edges are half-circles joining 
$\frac{r}{s}$ and $\frac{r'}{s'}$ whenever $rs'-r's = \pm 1 $. Every triangle is of the following form
\begin{center}
\begin{tikzpicture}
[scale=0.8]
	\draw (6,0) arc (0:180:2);
	\draw (6,0) node[below] {$\frac{r'}{s'}$};
	\draw (2,0) node[below] {$\frac{r}{s}$};
	\draw (6,0) arc (0:180:1);
	\draw (4,0) arc (0:180:1);
	\draw (4,0) node[below] {$\frac{r+r'}{s+s'}$};
\end{tikzpicture}
\end{center}
The \textit{Farey sum} of two rationals $\frac{r}{s}, \frac{r'}{s'}$ is the rational $\frac{r+r'}{s+s'}$ appearing as the median vertex in the triangle.

One defines the $q$-analogue $\left[\frac{r}{s}\right]_q$ of the rational  $\frac{r}{s}$ using the structure of the Farey tessellation. First, one assigns a weight, which is a power of $q$, to each edges of the triangles except for the one joining $\frac01$ and $\frac10$. Then one uses a $q$-deformation of the Farey sum involving the weights of the triangles.
The starting point is given by the triangle 
\begin{center}
\begin{tikzpicture}
[scale=0.8]
	\draw (6,0) arc (0:180:2);
	\draw (6,0) node[below] {$\frac{1}{0}$};
	\draw (2,0) node[below] {$\frac{0}{1}$};
	\draw (6,0) arc (0:180:1);
	\draw (4,0) arc (0:180:1);
	\draw (4,0) node[below] {$\frac{1}{1}$};
	
	\draw (4,2) node[above, scale=0.7] {$$};
	\draw (5,1) node[above, scale=0.7] {$1$};
	\draw (3,1) node[above, scale=0.7] {$1$};
\end{tikzpicture}
\end{center}
and the picture is completed recursively with the following local rule:
\begin{equation}\label{locr}
\,
\end{equation}
\begin{center}
\begin{tikzpicture}
[scale=0.8]
\vspace{-1cm}
	\draw (0,0) arc (0:180:2);
	\draw (0,0) node[below] {$\frac{\Rc'}{\Sc'}$};
	\draw (-4,0) node[below] {$\frac{\Rc}{\Sc}$};
	\draw (0,0) arc (0:180:1);
	\draw (-2,0) arc (0:180:1);
	\draw (-2,0) node[below] {$\frac{\Rc+q^d \Rc'}{\Sc+q^d \Sc'}$};

	\draw (-1,1) node[above, scale=0.7] {$q^d$};
	\draw (-3,1) node[above, scale=0.7] {$1$};
	\draw (-2,2) node[above, scale=0.7] {$q^{d-1}$};
\end{tikzpicture}
\end{center}

This process assigns rational functions $\frac{\Rc}{\Sc}$ to each vertices $\frac{r}{s}$. These are by definition the $q$-analogues $\left[\frac{r}{s}\right]_q$ of the rationals as introduced in  \cite{MGOfmsigma}. 

Figure \ref{qFarey} gives the first steps of the process. The next step would be to add the median points between all consecutive rational points already appearing in the picture. For instance the next step would give $\left[\frac{7}{5}\right]_q$ as the mediant point of $\left[\frac{4}{3}\right]_q$ and $\left[\frac{3}{2}\right]_q$:
$$
 \left[\frac{7}{5}\right]_q=\frac{(1+q+q^{2}+q^{3})+q^{2}(1+q+q^{2})}{(1+q+q^{2})+q^{2}(1+q)}=
 \frac{1+q+2q^{2}+2q^{3}+q^{4}}{1+q+2q^{2}+q^{3}}
$$

\subsection{First properties of $q$-rationals}
We give elementary properties of the polynomials appearing in the denominator and numerator of $\left[\frac{r}{s}\right]_q=\frac{\Rc}{\Sc}$. Note that $\Rc$ and $\Sc$ can be computed recursively independently one from the other. One has the following properties
\begin{itemize}
\item $\Rc$ and $\Sc$ are coprime polynomials in $q$;
\item they have positive integer coefficients;
\item the coefficients of the lowest and of the highest degree terms are equal to 1;
\item the sequences of coefficients are unimodal (this property was conjectured in \cite{MGOfmsigma}, proved in particular cases in \cite{McCSS}, and finally proved in full generality in \cite{OgRa}).
\end{itemize}

In addition when $\frac{r}{s}> 1$ both polynomials $\Rc$ and $\Sc$ in the $q$-deformation have a constant term equal to 1. When $\frac{r}{s}< 1$  a power of $q$ can be factored out of $\Rc$ (see \cite[Prop 2.4]{LMGadv} for a precise formula).

Furthermore when $\frac{r}{s}> 1$ there is a unique continued fraction expansion of the form \eqref{CFsEq1} with positive coefficients $a_{i}$ and even length $n$.
Enumerative interpretations for $\Rc$ and $\Sc$ have been given using different combinatorial models encoded by the sequence of positive coefficients $a_{i}$ (e.g. using closures of graphs \cite{MGOfmsigma}, poset order ideals \cite{McCSS}, snake graphs \cite{Ove}).

We present in the next subsection enumerative interpretations for $\Rc$ and $\Sc$ using paths in the Farey tesselation.

\begin{figure}
\begin{tikzpicture}
[scale=0.4]

	\draw (-19,0) node[below] {$\frac{0}{1}$};
	\draw (17,0) arc (0:180:18);
	\draw (-15,0) arc (0:180:2);
	
	\draw (-16.6,0) node[below] {$\left[\frac{1}{2}\right]_{q}$};
	\draw (-15,0) arc (0:180:1);
	\draw (-17,0) arc (0:180:1);
	\draw(-17,2) node[above, scale=0.6]{$1$};
	\draw(-16,1) node[above, scale=0.55]{$q$};
	\draw(-18,1) node[above, scale=0.55]{$1$};

	\draw (17,0) node[below] {$\frac{1}{0}$};
	\draw (17,0) arc (0:180:16);
	\draw (-15,0) node[below] {$\frac{1}{1}$};
	
	\draw (1,0) arc (0:180:8);
	\draw (1.4,0) node[below] {$\left[\frac{2}{1}\right]_{q}$};
	\draw(1,16) node[above, scale=0.7]{$1$};
	\draw (17,0) arc (0:180:8);
	
	\draw (-7,0) arc (0:180:4);
	\draw (-6.6,0) node[below] {$\left[\frac{3}{2}\right]_{q}$};
	\draw(-7,8) node[above, scale=0.7]{$1$};
	\draw (1,0) arc (0:180:4);
	
	\draw (9,0) arc (0:180:4);
	\draw (9.4,0) node[below] {$\left[\frac{3}{1}\right]_{q}$};
	\draw(9,8) node[above, scale=0.7]{$q$};
	\draw (17,0) arc (0:180:4);
	
	\draw (-11,0) arc (0:180:2);
	\draw (-10.6,0) node[below] {$\left[\frac{4}{3}\right]_{q}$};
	\draw(-11,4) node[above, scale=0.7]{$1$};
	\draw (-7,0) arc (0:180:2);
	\draw(-13,2) node[above, scale=0.6]{$1$};
	\draw(-9,2) node[above, scale=0.6]{$q$};
	
	\draw (-3,0) arc (0:180:2);
	\draw (-2.6,0) node[below] {$\left[\frac{5}{3}\right]_{q}$};
	\draw(-3,4) node[above, scale=0.7]{$q$};
	\draw (1,0) arc (0:180:2);
	\draw(-5,2) node[above, scale=0.6]{$1$};
	\draw(-1,2) node[above, scale=0.6]{$q^2$};
	
	\draw (5,0) arc (0:180:2);
	\draw (5.4,0) node[below] {$\left[\frac{5}{2}\right]_{q}$};
	\draw(5,4) node[above, scale=0.7]{$1$};
	\draw (9,0) arc (0:180:2);
	\draw(3,2) node[above, scale=0.6]{$1$};
	\draw(7,2) node[above, scale=0.6]{$q$};
	
	\draw (13,0) arc (0:180:2);
	\draw (13.4,0) node[below] {$\left[\frac{4}{1}\right]_{q}$};
	\draw(13,4) node[above, scale=0.7]{$q^2$};
	\draw (17,0) arc (0:180:2);
	\draw(11,2) node[above, scale=0.6]{$1$};
	\draw(15,2) node[above, scale=0.6]{$q^3$};

	\draw (1.4,-1.2) node[below] {$\rotatebox[origin=c]{90}{$=$}$};
	\draw (1.4,-2) node[below] {$\frac{1+q}{1}$};

	\draw (-6.6,-1.2) node[below] {$\rotatebox[origin=c]{90}{$=$}$};
	\draw (-6.6,-2) node[below] {$\frac{1+q+q^2}{1+q}$};

	\draw (9.4,-1.2) node[below] {$\rotatebox[origin=c]{90}{$=$}$};
	\draw (9.4,-2) node[below] {$\frac{1+q+q^2}{1}$};

	\draw (-10.6,-1.2) node[below] {$\rotatebox[origin=c]{90}{$=$}$};
	\draw (-10.6,-2) node[below] {$\frac{1+q+q^2+q^3}{1+q+q^2}$};
	
	\draw (-2.6,-1.2) node[below] {$\rotatebox[origin=c]{90}{$=$}$};
	\draw (-2.6,-2) node[below] {$\frac{1+q+2q^2+q^3}{1+q+q^2}$};

	\draw (5.4,-1.2) node[below] {$\rotatebox[origin=c]{90}{$=$}$};
	\draw (5.4,-2) node[below] {$\frac{1+2q+q^2+q^3}{1+q}$};

	\draw (13.4,-1.2) node[below] {$\rotatebox[origin=c]{90}{$=$}$};
	\draw (13.4,-2) node[below] {$\frac{1+q+q^2+q^3}{1}$};

	\draw (-16.6,-1.2) node[below] {$\rotatebox[origin=c]{90}{$=$}$};
	\draw (-16.6,-2) node[below] {$\frac{q}{1+q}$};
		
	\draw[dashed] (-13,0) arc (0:180:1);
	\draw[dashed] (-11,0) arc (0:180:1);
	\draw[dashed] (-9,0) arc (0:180:1);
	\draw[dashed] (-7,0) arc (0:180:1);
	\draw[dashed] (-5,0) arc (0:180:1);
	\draw[dashed] (-3,0) arc (0:180:1);
	\draw[dashed] (-1,0) arc (0:180:1);
	\draw[dashed] (1,0) arc (0:180:1);
	\draw[dashed] (3,0) arc (0:180:1);
	\draw[dashed] (5,0) arc (0:180:1);
	\draw[dashed] (7,0) arc (0:180:1);
	\draw[dashed] (9,0) arc (0:180:1);
	\draw[dashed] (11,0) arc (0:180:1);
	\draw[dashed] (13,0) arc (0:180:1);
	\draw[dashed] (15,0) arc (0:180:1);
	\draw[dashed] (17,0) arc (0:180:1);
\end{tikzpicture}
\caption{Upper part of the Farey tessellation with weights carried by the edges and $q$-deformed rationals labeling the vertices.}\label{qFarey}
\end{figure}
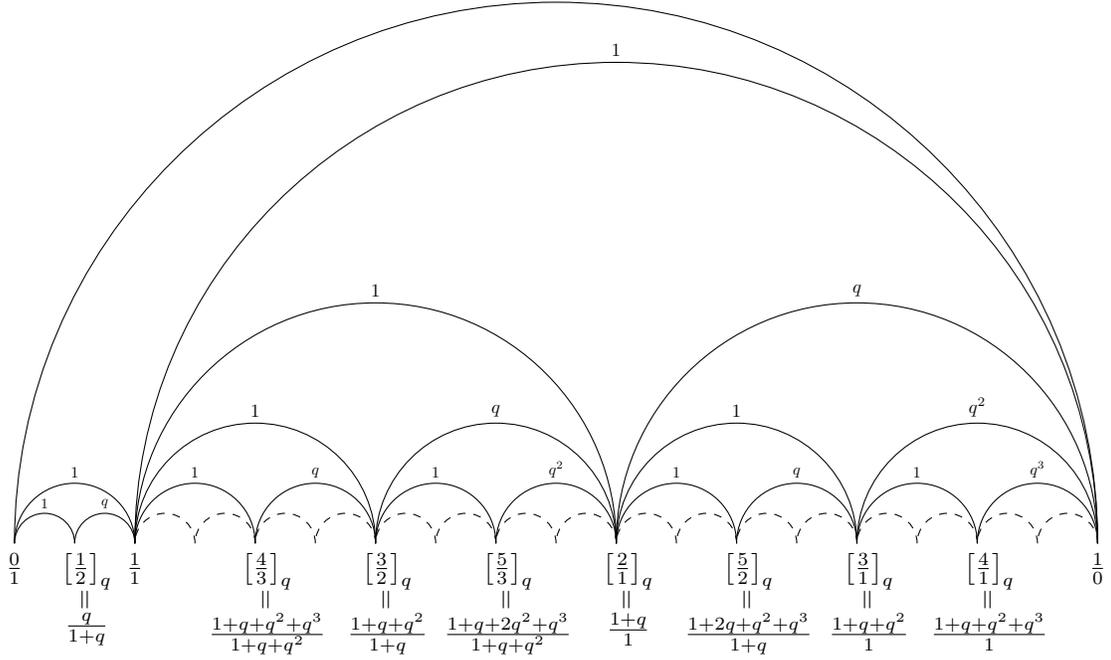

\subsection{Paths in the Farey tessellation}\label{pat}

We assign an orientation of each edges of the Farey tessellation, except for the edge joining $\frac01$ and $\frac10$, so that the following local rule holds in every triangle:
\begin{center}
\begin{tikzpicture}
[scale=0.8]
	\draw (6,0) arc (0:180:2);
	\draw (6,0) arc (0:180:1);
	\draw (4,0) arc (0:180:1);
		\draw [->,line width=2pt, red] (5,1)-- (5.1,1) ;
	\draw [<-,line width=2pt, red] (3,1)-- (3.1,1) ;
\end{tikzpicture}
\end{center}
Except for the vertices labeled by $\frac01$ and $\frac10$, every vertex is the median point of a unique triangle, so that at each vertex there are exactly two outgoing arrows, one oriented to the left and the other to the right.

A \textit{path} in the Farey tessellation is a sequence 
$$\pi : \frac{r_0}{s_0} \xrightarrow{\rho_1} \frac{r_1}{s_1} \xrightarrow{\rho_2}  \ldots \xrightarrow{\rho_{n-1}} \frac{r_{n-1}}{s_{n-1}} \xrightarrow{\rho_n} \frac{r_n}{s_n}$$ such that $\rho_i $ is an edge oriented from $ \frac{r_{i-1}}{s_{i-1}}$  to $\frac{r_i}{s_i} $. We will write for short $\pi: \frac{r_0}{s_0}\to \frac{r_n}{s_n}$ .

We denote by $\wt(\rho)$ the weight assigned to the edge $\rho$ in Section \ref{farat}. We define the \textit{weight} of the path $\pi : \frac{r_0}{s_0} \xrightarrow{\rho_1} \frac{r_1}{s_1} \xrightarrow{\rho_2} \ldots \xrightarrow{\rho_{n-1}} \frac{r_{n-1}}{s_{n-1}} \xrightarrow{\rho_n} \frac{r_n}{s_n}$ by the product
\begin{eqnarray}\label{wtdef}
\wt(\pi):=\wt(\rho_1)\wt(\rho_2)\cdots \wt(\rho_n).
\end{eqnarray}

Let $\frac{r}{s}$ be a rational greater than 1.

 We define the \textit{right-path} of $\frac{r}{s}$ as the shortest path  (in terms of numbers of edges involved) from $\frac{r}{s}$  to $\frac10$, starting with the edge oriented to the right. This path uses only edges oriented to the right which have weight a positive power of $q$. 
 
Similarly, we define the \textit{left-path} of $\frac{r}{s}$ as the shortest path  from $\frac{r}{s}$  to $\frac10$, starting with the edge oriented to the left. This path uses only edges of weight 1 which are all oriented to the left except for the last one joining $\frac11$ to $\frac10$ which is oriented to the right.

Finally we define the \textit{area} and \textit{coarea} of a path $\pi$ starting at $\frac{r}{s}$ as
\begin{eqnarray}\label{ardef}
\ar(\pi)&:=&\#\{\text{triangles enclosed between $\pi$ and the right-path of  $\textstyle\frac{r}{s}$ }\}\\
\label{coardef}
\coar(\pi)&:=&\#\{\text{triangles enclosed between $\pi$ and the left-path of  $\textstyle\frac{r}{s}$ }\}
\end{eqnarray}

\begin{ex}\label{triang75}
For example, in the case of $\frac75$ the left-path is $\frac75 \to \frac43 \to \frac11\to \frac 10$ and the right-path is $\frac75 \to \frac32 \to \frac21\to \frac 10$. They are drawn in orange and blue respectively in the following picture.
\begin{center}
\begin{tikzpicture}
[scale=0.2]
	\begin{small}
	\draw (17,0) node[below] {$\frac{1}{0}$};
	\draw [orange, very thick] (17,0) arc (0:180:16);
	\draw (-15,0) node[below] {$\frac{1}{1}$};
	
	\draw (1,0) arc (0:180:8);
	\draw (1,0) node[below] {$\frac{2}{1}$};
	\draw [blue, very thick] (17,0) arc (0:180:8);
	
	\draw (-7,0) arc (0:180:4);
	\draw (-7,0) node[below] {$\frac{3}{2}$};
	\draw [blue, very thick] (1,0) arc (0:180:4);
	
	\draw [orange, very thick, fill=white] (-11,0) arc (0:180:2);
	\draw (-11,0) node[below] {$\frac{4}{3}$};
	\draw (-7,0) arc (0:180:2);

	\draw [orange, very thick, fill=white] (-9,0) arc (0:180:1);
	\draw (-9,0) node[below] {$\frac{7}{5}$};
	\draw [blue, very thick] (-7,0) arc (0:180:1);

	\draw (1,16) node[above, scale=0.7] {$1$};
	\draw [<-,line width=0.9pt, red] (1,16)-- (0.99,16) ;
	\draw (-7,8) node[above,scale=0.7] {$1$};
	\draw [<-,line width=0.9pt, red] (-7,8)-- (-6.99,8) ;
	\draw (9,8) node[above, scale=0.7] {$q$};
	\draw [<-,line width=0.9pt, red] (9,8)-- (8.99,8) ;
	\draw (-11,4) node[above, scale=0.7] {$1$};
	\draw [<-,line width=0.9pt, red] (-11,4)-- (-10.99,4) ;
	\draw (-3,4) node[above, scale=0.7] {$q$};
	\draw [->,line width=0.9pt, red] (-3,4)-- (-2.99,4) ;
	\draw (-13,2) node[above, scale=0.7] {$1$};
	\draw [<-,line width=0.9pt, red] (-13,2)-- (-12.99,2) ;
	\draw (-9.2,2) node[above,scale=0.7] {$q$};
	\draw [->,line width=0.9pt, red] (-9,2)-- (-8.99,2) ;
	\draw (-10,1) node[below, scale=0.55] {$1$};
	\draw [<-,line width=0.9pt, red] (-10,1)-- (-9.99,1) ;
	\draw (-8,1) node[below, scale=0.55] {$q^2$};
	\draw [->,line width=0.9pt, red] (-8,1)-- (-7.99,1) ;
	\end{small}
\end{tikzpicture}

\end{center}

For our purpose we will only consider paths ending at $\frac11$ or $\frac10$. Examples of such paths and their corresponding areas/coareas are given in Figure \ref{ex5} and \ref{ex7}.

\end{ex}

\subsection{Enumerative interpretations of the $q$-rationals in the Farey tessellation}
We are now ready to formulate two enumerative interpretations of the $q$-rationals in terms of paths in the Farey tessellation. 
The proofs are given in Section \ref{Pf}.
The first interpretation requires the weight assignment on the edges of the tessellation introduced in \S\ref{farat}. 

\begin{thm} \label{thwt}
Let $\frac{r}{s}$ be a rational greater than 1 and let $\frac{\Rc}{\Sc}=\left[\frac{r}{s}\right]_q$ be its $q$-deformation. One has
\begin{eqnarray*}
\Rc&=&\sum_{\pi:\frac{r}{s}\to \frac10}\wt(\pi)\;,\\
\Sc&=&\sum_{\pi:\frac{r}{s}\to \frac11}\wt(\pi)\;,\\
\end{eqnarray*}
where $\wt$ is the weight function defined in \eqref{wtdef}.
\end{thm}

The second interpretation realizes the numerators and denominators of the $q$-rationals as the generating functions for the area of the paths.
\begin{thm} \label{thar}
Let $\frac{r}{s}$ be a rational greater than 1 and let $\frac{\Rc}{\Sc}=\left[\frac{r}{s}\right]_q$ be its $q$-deformation. One has
\begin{eqnarray*}
\Rc&=&\sum_{\pi:\frac{r}{s}\to \frac10}q^{\coar(\pi)}\;,\\
\Sc&=&\sum_{\pi:\frac{r}{s}\to \frac11}q^{\coar(\pi)}\; ,\\
\end{eqnarray*}
where $\coar$ is the coarea of the path defined in \eqref{coardef}.
\end{thm}

Note  that for $q=1$ the theorems lead to the following immediate corollary.
\begin{cor}
 Let $\frac{r}{s}$ be a rational greater than 1. In the oriented Farey tessellation $r$ is the total number of paths from $\frac{r}{s}$ to $\frac10$ and $s$ is the total number of paths from $\frac{r}{s}$ to $\frac11$.
\end{cor}

\begin{ex}Consider the case of 
$$
 \left[\frac{7}{5}\right]_q=
 \frac{1+q+2q^{2}+2q^{3}+q^{4}}{1+q+2q^{2}+q^{3}}.
$$
All the paths from $\frac75$ to $\frac11$ or to $\frac10$ are depicted in Figure \ref{ex5}. One can check that the power in the weight of a path coincides with the coarea of the path and they produce the polynomials in the ratio of $\left[\frac{7}{5}\right]_q$. Also note that any path ending at $\frac11$ can be extended to $\frac10$ just by including the top
edge, and this does not change the area.

\end{ex}


\begin{figure}[h!]
\begin{tikzpicture}
[scale=0.15]

	\draw (17,0) arc (0:180:16);
	\draw (-15,0) node[below] {$\frac{1}{1}$};
	
	\draw (1,0) arc (0:180:8);
	\draw (17,0) arc (0:180:8);
	
	\draw (-7,0) arc (0:180:4);
	\draw (1,0) arc (0:180:4);
	
	\draw [color=green, very thick] (-11,0) arc (0:180:2);
	\draw [very thin] (-11,0) arc (0:180:2);
	\draw (-7,0) arc (0:180:2);

	\draw [green,, very thick] (-9,0) arc (0:180:1);
	\draw [very thin] (-9,0) arc (0:180:1);

	\draw (-9,0) node[below] {$\frac{7}{5}$};
	\draw (-7,0) arc (0:180:1);
	
	\draw [<-,line width=0.8pt, red] (1,16)-- (0.99,16) ;
	\draw [<-,line width=0.8pt, red] (-7,8)-- (-6.99,8) ;
	\draw [<-,line width=0.8pt, red] (9,8)-- (8.99,8) ;
	\draw [<-,line width=0.8pt, red] (-11,4)-- (-10.99,4) ;
	\draw [->,line width=0.8pt, red] (-3,4)-- (-2.99,4) ;
	\draw [<-,line width=0.8pt, red] (-13,2)-- (-12.99,2) ;
	\draw [->,line width=0.8pt, red] (-8.75,2)-- (-8.7,2) ;
	\draw [<-,line width=0.8pt, red] (-10,1)-- (-9.99,1) ;
	\draw [->,line width=0.8pt, red] (-8,1)-- (-7.99,1) ;
\end{tikzpicture}
\begin{tikzpicture}
[scale=0.15]

	\draw (17,0) arc (0:180:16);
	\draw (-15,0) node[below] {$\frac{1}{1}$};
	
	\draw (1,0) arc (0:180:8);
	\draw (17,0) arc (0:180:8);
	
	\draw[green, fill=gray!30, very thick]  (-7,0) arc (0:180:4);
	\draw[very thin]  (-7,0) arc (0:180:4);

	\draw (1,0) arc (0:180:4);
	
	\draw[fill=white] (-11,0) arc (0:180:2);
	\draw[green,, very thick, fill=white]  (-7,0) arc (0:180:2);
	\draw[very thin]  (-7,0) arc (0:180:2);

	\draw[green,, very thick, fill=white]  (-9,0) arc (0:180:1);
	\draw[very thin]  (-9,0) arc (0:180:1);
	\draw (-9,0) node[below] {$\frac{7}{5}$};
	\draw[fill=white] (-7,0) arc (0:180:1);
	
	\draw [<-,line width=0.8pt, red] (1,16)-- (0.99,16) ;
	\draw [<-,line width=0.8pt, red] (-7,8)-- (-6.99,8) ;
	\draw [<-,line width=0.8pt, red] (9,8)-- (8.99,8) ;
	\draw [<-,line width=0.8pt, red] (-11,4)-- (-10.99,4) ;
	\draw [->,line width=0.8pt, red] (-3,4)-- (-2.99,4) ;
	\draw [<-,line width=0.8pt, red] (-13,2)-- (-12.99,2) ;
	\draw [->,line width=0.8pt, red] (-8.75,2)-- (-8.7,2) ;
	\draw [<-,line width=0.8pt, red] (-10,1)-- (-9.99,1) ;
	\draw [->,line width=0.8pt, red] (-8,1)-- (-7.99,1) ;
\end{tikzpicture}
\begin{tikzpicture}
[scale=0.15]

	\draw (17,0) arc (0:180:16);
	\draw (-15,0) node[below] {$\frac{1}{1}$};
	
	\draw[green,, very thick, fill=gray!30]  (1,0) arc (0:180:8);
	\draw[very thin]  (1,0) arc (0:180:8);

	\draw (17,0) arc (0:180:8);
	
	\draw (-7,0) arc (0:180:4);
	\draw[green,, very thick, fill=white]  (1,0) arc (0:180:4);
	\draw[very thin]  (1,0) arc (0:180:4);

	\draw[fill=white] (-11,0) arc (0:180:2);
	\draw[green,, very thick, fill=white]  (-7,0) arc (0:180:2);
	\draw[very thin]  (-7,0) arc (0:180:2);

	\draw [green,, very thick] (-9,0) arc (0:180:1);
	\draw [very thin] (-9,0) arc (0:180:1);

	\draw (-9,0) node[below] {$\frac{7}{5}$};
	\draw (-7,0) arc (0:180:1);
	
	\draw [<-,line width=0.8pt, red] (1,16)-- (0.99,16) ;
	\draw [<-,line width=0.8pt, red] (-7,8)-- (-6.99,8) ;
	\draw [<-,line width=0.8pt, red] (9,8)-- (8.99,8) ;
	\draw [<-,line width=0.8pt, red] (-11,4)-- (-10.99,4) ;
	\draw [->,line width=0.8pt, red] (-3,4)-- (-2.99,4) ;
	\draw [<-,line width=0.8pt, red] (-13,2)-- (-12.99,2) ;
	\draw [->,line width=0.8pt, red] (-8.75,2)-- (-8.7,2) ;
	\draw [<-,line width=0.8pt, red] (-10,1)-- (-9.99,1) ;
	\draw [->,line width=0.8pt, red] (-8,1)-- (-7.99,1) ;

\end{tikzpicture}
\vspace*{0.5cm}

\begin{tikzpicture}
[scale=0.15]

	\draw (17,0) arc (0:180:16);
	\draw (-15,0) node[below] {$\frac{1}{1}$};
	
	\draw (1,0) arc (0:180:8);
	\draw (17,0) arc (0:180:8);
	
	\draw[green, very thick, fill=gray!30]  (-7,0) arc (0:180:4);
	\draw[very thin]  (-7,0) arc (0:180:4);

	\draw (1,0) arc (0:180:4);
	
	\draw[fill=white] (-11,0) arc (0:180:2);
	\draw (-7,0) arc (0:180:2);

	\draw[fill=white] (-9,0) arc (0:180:1);
	\draw (-9,0) node[below] {$\frac{7}{5}$};
	\draw[green,, very thick, fill=white]  (-7,0) arc (0:180:1);
	\draw[very thin]  (-7,0) arc (0:180:1);

	\draw [<-,line width=0.8pt, red] (1,16)-- (0.99,16) ;
	\draw [<-,line width=0.8pt, red] (-7,8)-- (-6.99,8) ;
	\draw [<-,line width=0.8pt, red] (9,8)-- (8.99,8) ;
	\draw [<-,line width=0.8pt, red] (-11,4)-- (-10.99,4) ;
	\draw [->,line width=0.8pt, red] (-3,4)-- (-2.99,4) ;
	\draw [<-,line width=0.8pt, red] (-13,2)-- (-12.99,2) ;
	\draw [->,line width=0.8pt, red] (-8.75,2)-- (-8.7,2) ;
	\draw [<-,line width=0.8pt, red] (-10,1)-- (-9.99,1) ;
	\draw [->,line width=0.8pt, red] (-8,1)-- (-7.99,1) ;
\end{tikzpicture}
\begin{tikzpicture}
[scale=0.15]

	\draw (17,0) arc (0:180:16);
	\draw (-15,0) node[below] {$\frac{1}{1}$};
	
	\draw[green,, very thick, fill=gray!30]  (1,0) arc (0:180:8);
	\draw[very thin]  (1,0) arc (0:180:8);

	\draw (17,0) arc (0:180:8);
	
	\draw (-7,0) arc (0:180:4);
	\draw [green,, very thick, fill=white] (1,0) arc (0:180:4);
	\draw [very thin] (1,0) arc (0:180:4);

	\draw[fill=white] (-11,0) arc (0:180:2);
	\draw (-7,0) arc (0:180:2);

	\draw [fill=white](-9,0) arc (0:180:1);
	\draw (-9,0) node[below] {$\frac{7}{5}$};
	\draw[green,, very thick, fill=white]  (-7,0) arc (0:180:1);
	\draw[very thin]  (-7,0) arc (0:180:1);

	\draw [<-,line width=0.8pt, red] (1,16)-- (0.99,16) ;
	\draw [<-,line width=0.8pt, red] (-7,8)-- (-6.99,8) ;
	\draw [<-,line width=0.8pt, red] (9,8)-- (8.99,8) ;
	\draw [<-,line width=0.8pt, red] (-11,4)-- (-10.99,4) ;
	\draw [->,line width=0.8pt, red] (-3,4)-- (-2.99,4) ;
	\draw [<-,line width=0.8pt, red] (-13,2)-- (-12.99,2) ;
	\draw [->,line width=0.8pt, red] (-8.75,2)-- (-8.7,2) ;
	\draw [<-,line width=0.8pt, red] (-10,1)-- (-9.99,1) ;
	\draw [->,line width=0.8pt, red] (-8,1)-- (-7.99,1) ;
\end{tikzpicture}
\caption{The 5 paths (in green) from $\frac75$ to $\frac11$. The triangles shaded in gray  
 count for the coarea of the path whereas the triangles left blank count for the area.
The coarea-generating polynomial $1+q+2q^2+q^3$ corresponds to the denominator of~$\left[\frac75\right]_q$. }\label{ex5}
\end{figure}
\begin{figure}[h!]
\vspace*{1.5cm}

\begin{tikzpicture}
[scale=0.15]

	\draw (17,0) node[below] {$\frac{1}{0}$};
	\draw[green,, very thick]  (17,0) arc (0:180:16);
	\draw[very thin]  (17,0) arc (0:180:16);

	
	\draw (1,0) arc (0:180:8);
	\draw (17,0) arc (0:180:8);
	
	\draw (-7,0) arc (0:180:4);
	\draw (1,0) arc (0:180:4);
	
	\draw[green,, very thick]  (-11,0) arc (0:180:2);
	\draw[very thin]  (-11,0) arc (0:180:2);

	\draw (-7,0) arc (0:180:2);

	\draw[green,, very thick]  (-9,0) arc (0:180:1);
	\draw[very thin]  (-9,0) arc (0:180:1);

	\draw (-9,0) node[below] {$\frac{7}{5}$};
	\draw (-7,0) arc (0:180:1);
	
	\draw [<-,line width=0.8pt, red] (1,16)-- (0.99,16) ;
	\draw [<-,line width=0.8pt, red] (-7,8)-- (-6.99,8) ;
	\draw [<-,line width=0.8pt, red] (9,8)-- (8.99,8) ;
	\draw [<-,line width=0.8pt, red] (-11,4)-- (-10.99,4) ;
	\draw [->,line width=0.8pt, red] (-3,4)-- (-2.99,4) ;
	\draw [<-,line width=0.8pt, red] (-13,2)-- (-12.99,2) ;
	\draw [->,line width=0.8pt, red] (-8.75,2)-- (-8.7,2) ;
	\draw [<-,line width=0.8pt, red] (-10,1)-- (-9.99,1) ;
	\draw [->,line width=0.8pt, red] (-8,1)-- (-7.99,1) ;
\end{tikzpicture}
\begin{tikzpicture}
[scale=0.15]

	\draw (17,0) node[below] {$\frac{1}{0}$};
	\draw [green,, very thick] (17,0) arc (0:180:16);
	\draw [very thin] (17,0) arc (0:180:16);

	
	\draw (1,0) arc (0:180:8);
	\draw (17,0) arc (0:180:8);
	
	\draw[green,, very thick, fill=gray!30]  (-7,0) arc (0:180:4);
	\draw[very thin]  (-7,0) arc (0:180:4);

	\draw (1,0) arc (0:180:4);
	
	\draw[fill=white] (-11,0) arc (0:180:2);
	\draw [green,, very thick, fill=white] (-7,0) arc (0:180:2);
	\draw [very thin] (-7,0) arc (0:180:2);

	\draw [green,, very thick] (-9,0) arc (0:180:1);
	\draw [very thin] (-9,0) arc (0:180:1);

	\draw (-9,0) node[below] {$\frac{7}{5}$};
	\draw (-7,0) arc (0:180:1);
	
	\draw [<-,line width=0.8pt, red] (1,16)-- (0.99,16) ;
	\draw [<-,line width=0.8pt, red] (-7,8)-- (-6.99,8) ;
	\draw [<-,line width=0.8pt, red] (9,8)-- (8.99,8) ;
	\draw [<-,line width=0.8pt, red] (-11,4)-- (-10.99,4) ;
	\draw [->,line width=0.8pt, red] (-3,4)-- (-2.99,4) ;
	\draw [<-,line width=0.8pt, red] (-13,2)-- (-12.99,2) ;
	\draw [->,line width=0.8pt, red] (-8.75,2)-- (-8.7,2) ;
	\draw [<-,line width=0.8pt, red] (-10,1)-- (-9.99,1) ;
	\draw [->,line width=0.8pt, red] (-8,1)-- (-7.99,1) ;
\end{tikzpicture}
\begin{tikzpicture}
[scale=0.15]

	\draw (17,0) node[below] {$\frac{1}{0}$};
	\draw [green,, very thick] (17,0) arc (0:180:16);
	\draw [very thin] (17,0) arc (0:180:16);

	
	\draw (1,0) arc (0:180:8);
	\draw (17,0) arc (0:180:8);
	
	\draw[green,, very thick, fill=gray!30]  (-7,0) arc (0:180:4);
	\draw[very thin]  (-7,0) arc (0:180:4);

	\draw (1,0) arc (0:180:4);
	
	\draw[fill=white] (-11,0) arc (0:180:2);
	\draw (-7,0) arc (0:180:2);

	\draw [fill=white](-9,0) arc (0:180:1);
	\draw (-9,0) node[below] {$\frac{7}{5}$};
	\draw[green,, very thick, fill=white]  (-7,0) arc (0:180:1);
	\draw[very thin]  (-7,0) arc (0:180:1);

	\draw [<-,line width=0.8pt, red] (1,16)-- (0.99,16) ;
	\draw [<-,line width=0.8pt, red] (-7,8)-- (-6.99,8) ;
	\draw [<-,line width=0.8pt, red] (9,8)-- (8.99,8) ;
	\draw [<-,line width=0.8pt, red] (-11,4)-- (-10.99,4) ;
	\draw [->,line width=0.8pt, red] (-3,4)-- (-2.99,4) ;
	\draw [<-,line width=0.8pt, red] (-13,2)-- (-12.99,2) ;
	\draw [->,line width=0.8pt, red] (-8.75,2)-- (-8.7,2) ;
	\draw [<-,line width=0.8pt, red] (-10,1)-- (-9.99,1) ;
	\draw [->,line width=0.8pt, red] (-8,1)-- (-7.99,1) ;
\end{tikzpicture}
	\vspace*{0.5cm}

\begin{tikzpicture}
[scale=0.15]

	\draw (17,0) node[below] {$\frac{1}{0}$};
	\draw[green,, very thick]  (17,0) arc (0:180:16);
	\draw[very thin]  (17,0) arc (0:180:16);

	
	\draw[green,, very thick, fill=gray!30]  (1,0) arc (0:180:8);
	\draw[very thin]  (1,0) arc (0:180:8);

	\draw (17,0) arc (0:180:8);
	
	\draw (-7,0) arc (0:180:4);
	\draw[green,, very thick, fill=white]  (1,0) arc (0:180:4);
	\draw[very thin]  (1,0) arc (0:180:4);

	\draw[fill=white] (-11,0) arc (0:180:2);
	\draw[green,, very thick, fill=white]  (-7,0) arc (0:180:2);
	\draw[very thin]  (-7,0) arc (0:180:2);

	\draw[green,, very thick]  (-9,0) arc (0:180:1);
	\draw[very thin]  (-9,0) arc (0:180:1);

	\draw (-9,0) node[below] {$\frac{7}{5}$};
	\draw (-7,0) arc (0:180:1);
	
	\draw [<-,line width=0.8pt, red] (1,16)-- (0.99,16) ;
	\draw [<-,line width=0.8pt, red] (-7,8)-- (-6.99,8) ;
	\draw [<-,line width=0.8pt, red] (9,8)-- (8.99,8) ;
	\draw [<-,line width=0.8pt, red] (-11,4)-- (-10.99,4) ;
	\draw [->,line width=0.8pt, red] (-3,4)-- (-2.99,4) ;
	\draw [<-,line width=0.8pt, red] (-13,2)-- (-12.99,2) ;
	\draw [->,line width=0.8pt, red] (-8.75,2)-- (-8.7,2) ;
	\draw [<-,line width=0.8pt, red] (-10,1)-- (-9.99,1) ;
	\draw [->,line width=0.8pt, red] (-8,1)-- (-7.99,1) ;
\end{tikzpicture}
\begin{tikzpicture}
[scale=0.15]

	\draw (17,0) node[below] {$\frac{1}{0}$};
	\draw[green,, very thick]  (17,0) arc (0:180:16);
	\draw[very thin]  (17,0) arc (0:180:16);

	
	\draw[green,, very thick, fill=gray!30]  (1,0) arc (0:180:8);
	\draw[very thin]  (1,0) arc (0:180:8);

	\draw (17,0) arc (0:180:8);
	
	\draw (-7,0) arc (0:180:4);
	\draw[green,, very thick, fill=white]  (1,0) arc (0:180:4);
	\draw[very thin]  (1,0) arc (0:180:4);

	\draw[fill=white] (-11,0) arc (0:180:2);
	\draw (-7,0) arc (0:180:2);

	\draw [fill=white](-9,0) arc (0:180:1);
	\draw (-9,0) node[below] {$\frac{7}{5}$};
	\draw[green,, very thick, fill=white]  (-7,0) arc (0:180:1);
	\draw[very thin]  (-7,0) arc (0:180:1);

	\draw [<-,line width=0.8pt, red] (1,16)-- (0.99,16) ;
	\draw [<-,line width=0.8pt, red] (-7,8)-- (-6.99,8) ;
	\draw [<-,line width=0.8pt, red] (9,8)-- (8.99,8) ;
	\draw [<-,line width=0.8pt, red] (-11,4)-- (-10.99,4) ;
	\draw [->,line width=0.8pt, red] (-3,4)-- (-2.99,4) ;
	\draw [<-,line width=0.8pt, red] (-13,2)-- (-12.99,2) ;
	\draw [->,line width=0.8pt, red] (-8.75,2)-- (-8.7,2) ;
	\draw [<-,line width=0.8pt, red] (-10,1)-- (-9.99,1) ;
	\draw [->,line width=0.8pt, red] (-8,1)-- (-7.99,1) ;
\end{tikzpicture}
\begin{tikzpicture}
[scale=0.15]

	\draw (17,0) node[below] {$\frac{1}{0}$};
	\draw [fill=gray!30](17,0) arc (0:180:16);
	
	\draw (1,0) arc (0:180:8);
	\draw[green,, very thick, fill=white]  (17,0) arc (0:180:8);
	\draw[very thin]  (17,0) arc (0:180:8);

	\draw (-7,0) arc (0:180:4);
	\draw[green,, very thick, fill=white]  (1,0) arc (0:180:4);
	\draw[very thin]  (1,0) arc (0:180:4);

	\draw[fill=white] (-11,0) arc (0:180:2);
	\draw [green,, very thick, fill=white] (-7,0) arc (0:180:2);
	\draw [very thin] (-7,0) arc (0:180:2);

	\draw [green,, very thick] (-9,0) arc (0:180:1);
	\draw [very thin] (-9,0) arc (0:180:1);

	\draw (-9,0) node[below] {$\frac{7}{5}$};
	\draw (-7,0) arc (0:180:1);
	
	\draw [<-,line width=0.8pt, red] (1,16)-- (0.99,16) ;
	\draw [<-,line width=0.8pt, red] (-7,8)-- (-6.99,8) ;
	\draw [<-,line width=0.8pt, red] (9,8)-- (8.99,8) ;
	\draw [<-,line width=0.8pt, red] (-11,4)-- (-10.99,4) ;
	\draw [->,line width=0.8pt, red] (-3,4)-- (-2.99,4) ;
	\draw [<-,line width=0.8pt, red] (-13,2)-- (-12.99,2) ;
	\draw [->,line width=0.8pt, red] (-8.75,2)-- (-8.7,2) ;
	\draw [<-,line width=0.8pt, red] (-10,1)-- (-9.99,1) ;
	\draw [->,line width=0.8pt, red] (-8,1)-- (-7.99,1) ;
\end{tikzpicture}
	\vspace*{0.5cm}

\begin{tikzpicture}
[scale=0.15]

	\draw (17,0) node[below] {$\frac{1}{0}$};
	\draw[fill=gray!30] (17,0) arc (0:180:16);
	
	\draw (1,0) arc (0:180:8);
	\draw[green,, very thick, fill=white]  (17,0) arc (0:180:8);
	\draw[very thin]  (17,0) arc (0:180:8);

	\draw (-7,0) arc (0:180:4);
	\draw[green,, very thick, fill=white]  (1,0) arc (0:180:4);
	\draw[very thin]  (1,0) arc (0:180:4);

	\draw[fill=white]  (-11,0) arc (0:180:2);
	\draw (-7,0) arc (0:180:2);

	\draw[fill=white]  (-9,0) arc (0:180:1);
	\draw (-9,0) node[below] {$\frac{7}{5}$};
	\draw[green,, very thick, fill=white]  (-7,0) arc (0:180:1);
	\draw[very thin]  (-7,0) arc (0:180:1);

	\draw [<-,line width=0.8pt, red] (1,16)-- (0.99,16) ;
	\draw [<-,line width=0.8pt, red] (-7,8)-- (-6.99,8) ;
	\draw [<-,line width=0.8pt, red] (9,8)-- (8.99,8) ;
	\draw [<-,line width=0.8pt, red] (-11,4)-- (-10.99,4) ;
	\draw [->,line width=0.8pt, red] (-3,4)-- (-2.99,4) ;
	\draw [<-,line width=0.8pt, red] (-13,2)-- (-12.99,2) ;
	\draw [->,line width=0.8pt, red] (-8.75,2)-- (-8.7,2) ;
	\draw [<-,line width=0.8pt, red] (-10,1)-- (-9.99,1) ;
	\draw [->,line width=0.8pt, red] (-8,1)-- (-7.99,1) ;
\end{tikzpicture}

\caption{The 7 paths (in green) from $\frac75$ to $\frac10$. The triangles shaded in gray  
 count for the coarea of the path whereas the triangles left blank count for the area.
The coarea-generating polynomial $1+q+2q^2+2q^3+q^4$ corresponds to numerator of~$\left[\frac75\right]_q$. }\label{ex7}
\end{figure}


\subsection{Proofs of Theorems \ref{thwt} and \ref{thar}}\label{Pf}
 The theorems will be proved by induction. Suppose that the formula holds for $\frac{\Rc}{\Sc}=\left[\frac{r}{s}\right]_q$ and 
 $\frac{\Rc'}{\Sc'}=\left[\frac{r'}{s'}\right]_q$ where $\frac{r}{s}<\frac{r'}{s'}$ are two rationals linked by an edge in the Farey tessellation. The same formula will follow for $\frac{\Rc''}{\Sc''}=\left[\frac{r''}{s''}\right]_q$ where $\frac{r''}{s''}= \frac{r+r'}{s+s'}$ is the median due to the local rule \eqref{locr}. Indeed, a path starting at $\frac{r''}{s''}$ will either use the left edge of weight 1 and then a path starting at $\frac{r}{s}$ or it will use the right edge of weight $q^d$ and then a path starting at $\frac{r'}{s'}$. 
 
 Hence, one easily computes
\begin{equation*}
\begin{array}{lcccccc}\displaystyle
\sum_{\pi:\frac{r''}{s''}\to \frac10}{\wt(\pi)}&=&\displaystyle \sum_{\pi:\frac{r''}{s''}\to\frac{r}{s}\to\frac10}{\wt(\pi)}&+&\displaystyle\sum_{\pi:\frac{r''}{s''}\to\frac{r'}{s'}\to\frac10}{\wt(\pi)}\\[20pt]
&=&\displaystyle \sum_{\pi:\frac{r}{s}\to\frac10}{\wt(\pi)}&+&\displaystyle\sum_{\pi:\frac{r'}{s'}\to\frac10}q^d{\wt(\pi)}\\[20pt]
&=&\Rc&+&q^d \Rc'\\[10pt]
&=&\Rc''.
\end{array}
\end{equation*}
And similarly for $\Sc$. Theorem \ref{thwt} is proved.

Theorem \ref{thar} is proved in the same inductive way by counting triangles enclosed with respect to the left-paths. One notices that a path starting from $\frac{r''}{s''}$  and using the left edge of weight 1 will not enclose more triangles than the rest of the path starting at $\frac{r}{s}$. Whereas a path starting from $\frac{r''}{s''}$  and using the right edge of weight $q^d$ will always enclose $d$ more extra triangles than the rest of the path starting at $\frac{r'}{s'}$, located under the left-path of $\frac{r'}{s'}$ and the right edge of weight $q^d$, as shown in the following picture.

\begin{center}
\begin{tikzpicture}
[scale=0.4]
	\draw [draw=orange, fill=yellow!50,  very thick] (0,0) arc (0:180:8);
	\draw (0,0) node[below] {$\frac{\mathcal{R}'}{\mathcal{S}'}$};
	\draw (-16,0) node[below] {$\frac{\mathcal{R}_3}{\mathcal{S}_3}$};
	\draw [draw=blue, very thick ](0,0) arc (180:85:8);
	\draw [<-,line width=2pt, red] (-8,8)-- (-7.99,8);
	\draw [<-,line width=2pt, red] (8,8)-- (7.99,8);
	\draw (-8,8) node[above] {$1$};
	\draw (0,0) arc (0:180:4);
	\draw (-8,0) node[below] {$\frac{\mathcal{R}}{\mathcal{S}}$};
	\draw [draw=blue, fill=white, very thick](0,0) arc (0:180:2);
	\draw (-4,0) node[below] {$\frac{\mathcal{R}''}{\mathcal{S}''}$};
	\draw [draw=orange, fill=white, very thick](-4,0) arc (0:180:2);
	\draw (-4,3.7) node[above] {$q^{d-1}$};
	\draw (-2,2) node[above] {$q^d$};
	\draw (-6,2) node[above] {$1$};
	\draw [->,line width=2pt, red] (-4,4)-- (-3.99,4);
	\draw [->,line width=2pt, red] (-2,2)-- (-1.99,2);	
	\draw [<-,line width=2pt, red] (-6,2)-- (-5.99,2);	
	\draw  (0,0) arc (0:100:5);
	\draw (-4.5,4.75) node[above] {$\vdots$};
	\draw (0,0) arc (0:100:6.5);
	\draw (-4.5,6.2) node[above] {$q$};	
	\draw [<-,line width=2pt, red] (-4.5,5)-- (-4.51,5);
	\draw [<-,line width=2pt, red] (-4.5,6.2)-- (-4.51,6.2);	
\end{tikzpicture}
\end{center}

\section{$q$-continuants and triangulated polygons}\label{contsec}

By using continued fraction expansions of the rationals we reformulate the results of the previous section in terms of continuants and triangulations of polygons. We start by defining the $q$-analogues of the objects given in the introduction. These $q$-analogues already appeared in \cite[\S4.2, \S5.2]{MGOfmsigma}. 
\subsection{$q$-continuants}
The $q$-deformations of the continuants in \eqref{ContEq} are polynomials in $q^{\pm 1}$  defined by the following determinants of tridiagonal matrices 
\cite[\S5.2]{MGOfmsigma}:
\begin{equation}
\label{K+Eq}
K^{}_{n}(a_1,\ldots,a_{n})_{q}:=
q^{\sum_i a_{2i} -1}
\left|
\begin{array}{ccccccccc}
[a_1]_{q}&q^{a_{1}}&&\\[6pt]
-1&[a_{2}]_{q^{-1}}&q^{-a_{2}}&\\[4pt]
&-1&[a_3]_{q}&q^{a_{3}}&\\[4pt]
&&-1&[a_{4}]_{q^{-1}}&q^{-a_{4}}&\\[10pt]
&&&\ddots&\ddots&\!\!\ddots\\[10pt]
&&&&-1&\!\!\!\![a_{n}]_{q^{-1}}
\end{array}
\right|
\end{equation}
where $a_i$ are integers and $[a_i]_q$ are as in \eqref{qint} and $n$ is even, and
\begin{equation}
\label{KEq}
E_{k}(c_1,\ldots,c_k)_{q}:=
\left|
\begin{array}{cccccccc}
[c_1]_{q}&q^{c_{1}-1}&&&\\[6pt]
1&[c_{2}]_{q}&q^{c_{2}-1}&&\\[4pt]
&\ddots&\ddots&\!\!\ddots&\\[4pt]
&&1&\!\!\![c_{k-1}]_{q}&q^{c_{k-1}-1}\\[6pt]
&&&\!\!\!\!\!\!1&\!\!\!\!\!\!\!\![c_{k}]_{q}
\end{array}
\right|
\end{equation}
where $c_{i}$ are integers and $[c_i]_q$ are as in \eqref{qint}. 
We recall the following conventions:
$K^{}_{-1}()_{q}=E_{-1}()_q=0$ and $K^{}_{0}()_{q}=E_{0}()_q=1$. 
We also define $K^{}_{n-1}(a_2,\ldots,a_{n})_{q}$  by removing the first row and the first column in the determinant \eqref{K+Eq}.
When the coefficients $a_i$'s and $c_i$'s are positive integers the $q$-continuants are both polynomials in~$q$.

As in the classical case, $q$-continuants are related to continued fractions.
Let us use standard bracket notation for continued fractions: $[a_1,a_2,\ldots, a_{n}]$ stands for the regular continued fraction, i.e. the left hand side of \eqref{CFsEq1} and $\llbracket{}c_1,c_2,\ldots, c_{k}\rrbracket{}$  stands for the negative signed continued fraction, i.e. the left hand side of \eqref{CFsEq2}.
The $q$-analogues of \eqref{CFsEq1} and \eqref{CFsEq2}, when the coefficients $a_i$ and $c_i$ are integers, were introduced in \cite{MGOfmsigma} as follows
\begin{equation}\label{qregfrac}
\begin{array}{lcllcl}
[a_{1}, a_{2} \ldots, a_{n}]_{q}&:=&
[a_1]_{q} + \cfrac{q^{a_{1}}}{[a_2]_{q^{-1}} 
          + \cfrac{q^{-a_{2}}}{[a_{3}]_{q} 
          +\cfrac{q^{a_{3}}}{[a_{4}]_{q^{-1}}
          + \cfrac{q^{-a_{4}}}{
        {\ddots    +\cfrac{}{[a_n]_q^{\pm 1}}}    }
          } }} 
          &=&\dfrac{K_{n}(a_1,\ldots,a_n)_q}{K_{n-1}(a_2,\ldots,a_{n})_q}
          \end{array}
          \end{equation}
\begin{equation}\label{qnegfrac}
\begin{array}{lclcl}
 \llbracket{}c_1,c_2,\ldots, c_{k}\rrbracket{}_{q}&:=&
 [c_1]_{q} - \cfrac{q^{c_{1}-1}}{[c_2]_{q} 
          - \cfrac{q^{c_{2}-1}}{\ddots - \cfrac{q^{c_{k-1}-1}}{[c_{k}]_{q}} } }
          &=&\dfrac{E_{k}(c_1,\ldots,c_k)_q}{E_{k-1}(c_2,\ldots,c_{k})_q}.
\end{array}
\end{equation}

Finally, the $q$-continuants also appear
in $2\times 2$-matrices computed by multiplication of elementary matrices that belong to $\GL(2, \Z[q^{\pm 1}])$. One has
the following $q$-analogues of~\eqref{MatEq} according to \cite[\S4.2]{MGOfmsigma}. 
\begin{equation}
\label{qRegMat}
\begin{array}{lcl}
M^{+}(a_{1},\ldots, a_{n})_{q}&:=&q^{\sum_i a_{2i} }
\begin{pmatrix}
[a_{1}]_{q}&q^{a_{1}}\\[6pt]
1&0
\end{pmatrix}
\begin{pmatrix}
[a_{2}]_{q^{-1}}& q^{-a_{2}}\\[6pt]
1&0
\end{pmatrix}
\cdots
\begin{pmatrix}
[a_{n-1}]_{q}&q^{a_{n-1}}\\[6pt]
1&0
\end{pmatrix}
\begin{pmatrix}
[a_{n}]_{q^{-1}}&q^{-a_{n}}\\[6pt]
1&0
\end{pmatrix}\\[20pt]
&=&
\left(
\begin{array}{cc}
qK_{n}(a_1,\ldots,a_n)_q&\widetilde{K}_{n-1}(a_1,\ldots,a_{n-1})_q\\[10pt]
qK_{n-1}(a_2,\ldots,a_n)_q&\widetilde{K}_{n-2}(a_2,\ldots,a_{n-1})_q
\end{array}
\right)
\end{array}
\end{equation}
where the notation $\widetilde{K}$ stands for the mirror polynomial, i.e. the one with reversed sequence of coefficients,
and
\begin{equation}
\label{qNegMat}
\begin{array}{lcl}
M(c_{1},\ldots, c_{k})_{q}&:=&
\begin{pmatrix}
[c_{1}]_{q}&-q^{c_{1}-1}\\[6pt]
1&0
\end{pmatrix}
\begin{pmatrix}
[c_{2}]_{q}&-q^{c_{2}-1}\\[6pt]
1&0
\end{pmatrix}
\cdots
\begin{pmatrix}
[c_{k}]_{q}&-q^{c_{k}-1}\\[6pt]
1&0
\end{pmatrix}\\[20pt]
&=&
\begin{pmatrix}
E_{k}(c_{1},\ldots, c_{k})_q&-q^{c_{k}-1}E_{k-1}(c_{1},\ldots, c_{k-1})_q\\[10pt]
E_{k-1}(c_{2},\ldots, c_{k})_q&-q^{c_{k}-1}E_{k-2}(c_{2},\ldots, c_{k-1})_q
\end{pmatrix}
\end{array}.
\end{equation}

\subsection{$q$-rationals}
Every rational $\frac{r}{s}>1$ has canonical continued fraction expansions of the form $[a_{1}, a_{2} \ldots, a_{2m}]$  and 
$ \llbracket{}c_1,c_2,\ldots, c_{k}\rrbracket{}$ with integer coefficients $a_{i}\geq 1$ and $c_{i}\geq 2$.

\begin{thm}[\cite{MGOfmsigma}]\label{genfrac}
If $\frac{r}{s}=[a_{1}, \ldots, a_{2m}]_{}=\llbracket{}c_1,c_2,\ldots, c_{k}\rrbracket{}_{}$ are the canonical expansions of the rational $\frac{r}{s}>1$ then
$$\left[\frac{r}{s}\right]_{q}=[a_{1}, \ldots, a_{2m}]_{q}=\llbracket{}c_1,c_2,\ldots, c_{k}\rrbracket{}_{q}.$$
\end{thm}

The fact that the two $q$-continued fractions \eqref{qregfrac} and \eqref{qnegfrac} are the same whenever they coincide at $q=1$ is not obvious at first sight and was first proved in \cite{MGOfmsigma}. It turns out that this result can be extended to fractions with arbitrary integer coefficients, not only positive, see \cite[Thm 7]{LMGadv}. 

\subsection{Farey tessellation revisited}\label{Trs}
Let  $\frac{r}{s}$ be a rational greater than 1.
The positive integer coefficients appearing in the  canonical expansions $\frac{r}{s}=[a_{1}, \ldots, a_{2m}]_{}=\llbracket{}c_1,c_2,\ldots, c_{k}\rrbracket{}_{}$ have combinatorial interpretations in the Farey tessellation coming from \cite{Ser} and \cite{CoCo}. We refer to \cite{MGOfarey1} for a more detailed overview on the subject.

We will illustrate the statement with the running example of $\frac{7}{5}=[1,2,1,1]=\llbracket 2,2,3\rrbracket$.
\begin{equation}\label{T7/5far}
\begin{matrix}
\begin{tikzpicture}
[scale=0.25]
	\draw (17,0) node[below] {$\frac{1}{0}$};
	\draw [fill=pink!50] (17,0) arc (0:180:17);
	\draw [fill=blue!30] (17,0) arc (0:180:16);
	\draw [fill=white] (-15,0) arc (0:180:1);
	\draw (-15,0) node[below] {$\frac{1}{1}$};
	\draw (-17,0) node[below] {$\frac{0}{1}$};
	
	\draw (1,0) arc (0:180:8);
	\draw (1,0) node[below] {$\frac{2}{1}$};
	\draw  [fill=white] (17,0) arc (0:180:8);
	
	\draw  [fill=pink!50] (-7,0) arc (0:180:4);
	\draw (-7,0) node[below] {$\frac{3}{2}$};
	\draw  [fill=white] (1,0) arc (0:180:4);
	
	\draw  [fill=white]  (-11,0) arc (0:180:2);
	\draw (-11,0) node[below] {$\frac{4}{3}$};
	\draw [fill=blue!30]  (-7,0) arc (0:180:2);

	\draw  [fill=white] (-9,0) arc (0:180:1);
	\draw (-9,0) node[below] {$\frac{7}{5}$};
	\draw  [fill=white] (-7,0) arc (0:180:1);
	
	\draw [red] (-9,0)--(-9,17);		
\end{tikzpicture}
\end{matrix}
\end{equation}

One draws a vertical line passing through $\frac{r}{s}$. This line crosses finitely many triangles in the Farey tessellation. Keeping these triangles and removing the other ones one obtains a finite tessellation denoted by $\T_{\frac{r}{s}}$.

\textbf{Interpretation of $a_i$.}
From top to bottom the vertical line starts crossing $a_1$ adjacent triangles with the base at the left (in pink on the picture), then $a_2$ adjacent triangles with the base at the right (in blue on the picture), then $a_3$ triangles with the base at the left, and so on. There is an ambiguity for the last triangle (the one with median vertex $\frac{r}{s}$). We attach the last triangle in such a way that we have an even number of coefficients $a_i$.

In the example of $\frac{r}{s}=\frac75$, one counts $(a_1,a_2,a_3,a_4)=(1,2,1,1)$, which coincide with the coefficients in the regular continued fraction expansion $\frac75= [1,2,1,1]$.

\textbf{Interpretation of $c_i$.} These coefficients count the numbers of triangles in $\T_{\frac{r}{s}}$ incident to each vertex located at the right of $\frac{r}{s}$ and enumerated in decreasing order. In the example of $\T_{\frac{7}{5}}$ one counts $c_1=2$ triangles incident to vertex $\frac10$, $c_2=2$ triangles incident to $\frac21$ and $c_3=3$ triangles incident to $\frac32$, which coincide with the coefficients in the negative continued fraction expansion $\frac75=\llbracket 2,2,3\rrbracket$.

We will now simplify the picture and draw triangulations of convex (non strictly convex) polygons instead of triangulations in the Farey tessellation. 
For example, the Farey triangulation $\T_{\frac{7}{5}}$ in \eqref{T7/5far} will be depicted as the following triangulated convex heptagon:

\begin{equation}\label{T7/5euc}
\begin{matrix}
\begin{tikzpicture}[scale=0.5, baseline=(current  bounding  box.center) ]
\draw (2,0) node[below] {$\frac{0}{1}$};
\draw (5,0) node[below] {$\frac{1}{1}$};
\draw (8,0) node[below] {$\frac{4}{3}$};
\draw [fill=pink!50] (2,0)--(5, 0)--(4,4)--(2,0);
\draw [fill=blue!30] (4,4)--(8,4)--(5, 0);
\draw (6,4)--(5,0); 
\draw  [fill=pink!50]  (8,4)--(8,0)--(5,0);
\draw [fill=blue!30] (8,0)--(10,4)--(8,4);
\draw (10,4) node[above] {$\frac{7}{5}$};
\draw (4,4) node[above] {$\frac10$};
\draw (6,4) node[above] {$\frac21$};
\draw (8,4) node[above] {$\frac32$};
\end{tikzpicture}
\end{matrix}
\end{equation}
The edges in the triangulated polygon will inherit the orientation of the edges from the Farey tessellation defined in \S\ref{pat}, see also the next paragraph for nore details.

\subsection{Triangulated polygons}
Consider the triangulation $\T$ of a convex (non strictly convex) $n$-gon of the following form, sometimes called \textit{fan triangulation}:
\begin{equation}\label{trigen}
\begin{matrix}
\begin{tikzpicture}[scale=0.7]
\draw  [fill=pink!50]  (0,0)--(2,4)--(6,0);
\draw [fill=blue!30] (2,4)--(8,4)--(6,0);
\draw  [fill=pink!50]  (8,4)--(7.5,0)--(6,0);
\draw[white] (8,4)--(7.5,0);
\draw [fill=blue!30] (12,0)--(14,4)--(12.5,4);
\draw  [fill=pink!50]  (12,0)--(18,0)--(14,4);
\draw [fill=blue!30] (18,0)--(20,4)--(14,4);
\draw (0,0)--(8,0);
\draw (0,0)--(2,4);
\draw (2,4)--(8,4);
\draw[dashed] (8,4)--(12,4);
\draw[dashed] (8,0)--(12,0);
\draw (12,4)--(20,4);
\draw (18,0)--(20,4);
\draw (12,0)--(18,0);
\draw (2,4)--(2,0);
\draw (2,4)--(4,0);
\draw (2,4)--(6,0);
\draw (4,4)--(6,0);
\draw (6,4)--(6,0);
\draw (8,4)--(6,0);
\draw[dashed] (8.5,2)--(11.6,2);
\draw (14,4)--(12,0);
\draw (14,4)--(14,0);
\draw (14,4)--(16,0);
\draw (14,4)--(18,0);
\draw (16,4)--(18,0);
\draw (18,4)--(18,0);
\draw [<->] (2,4.7)-- (8,4.7);
\draw (5,4.7) node[above] {$a_2$};
\draw [->] (11.5,4.7)-- (13.9,4.7);
\draw (12.7,4.7) node[above] {$a_{2m-2}$};
\draw [<->] (14.1,4.7)-- (20,4.7);
\draw (17,4.7) node[above] {$a_{2m}$};
\draw [<->] (0,-0.7)-- (5.9,-0.7);
\draw (3,-0.7) node[below] {$a_1$};
\draw [<-] (6.1,-0.7)-- (8.5,-0.7);
\draw (7.3,-0.7) node[below] {$a_3$};
\draw [<->] (12,-0.7)-- (18,-0.7);
\draw (15,-0.7) node[below] {$a_{2m-1}$};
\draw (0,0) node[below] { ${}_0$};
\draw (2,0) node[below] { ${}_{n-1}$};
\draw (4,0) node[below] { ${}_{n-2}$};
\draw (18,0) node[below] { ${}_{k+2}$};
\draw (2,4) node[above] { ${}_1$};
\draw (4,4) node[above] { ${}_2$};
\draw (18,4) node[above] { ${}_k$};
\draw (20,4) node[above] { ${}_{k+1}$};
\draw [<-,red,thick] (1,0)-- (1.01,0);
\draw [<-,red,thick] (3,0)-- (3.01,0);
\draw [<-,red,thick] (5,0)-- (5.01,0);
\draw [<-,red,thick] (6.8,0)-- (6.81,0);
\draw [<-,red,thick] (13,0)-- (13.01,0);
\draw [<-,red,thick] (15,0)-- (15.01,0);
\draw [<-,red,thick] (17,0)-- (17.01,0);
\draw [<-,red,thick] (3,4)-- (3.01,4);
\draw [<-,red,thick] (5,4)-- (5.01,4);
\draw [<-,red,thick] (7,4)-- (7.01,4);
\draw [<-,red,thick] (13.1,4)-- (13.11,4);
\draw [<-,red,thick] (15,4)-- (15.01,4);
\draw [<-,red,thick] (17,4)-- (17.01,4);
\draw [<-,red,thick] (19,4)-- (19.01,4);
\draw [->,red,thick] (2,2)-- (2,2.01);
\draw [->,red,thick] (2.995,2)-- (2.99,2.01);
\draw [->,red,thick] (4.003,2)-- (3.995,2.01);
\draw [<-,red,thick] (5.003,2)-- (4.995,2.01);
\draw [<-,red,thick] (6,2)-- (6,2.01);
\draw [<-,red,thick] (7,2)-- (7.005,2.01);
\draw [<-,red,thick] (13,2)-- (13.005,2.01);
\draw [->,red,thick] (14,2)-- (14,2.01);
\draw [->,red,thick] (14.995,2)-- (14.99,2.01);
\draw [->,red,thick] (16.003,2)-- (15.995,2.01);
\draw [<-,red,thick] (18,2)-- (18,2.01);
\draw [<-,red,thick] (17.003,2)-- (16.995,2.01);
\draw [<-,red,thick] (19,2)-- (19.005,2.01);
\end{tikzpicture}
\end{matrix}
\end{equation}

We associate two sequences of integers
\begin{enumerate}
\item
The integers $(a_1,a_2,\ldots,a_{2m})$, with $a_i\geq 1$, count the number of adjacent triangles according to their position ``base down'' or ``base up''
i.e. the triangulation consists of $a_1$ adjacent triangles base down at the left, followed by $a_2$ triangles base up and so on:
\item
The integers $\left(c_0, c_1,\ldots, c_{n-1}\right)$ 
count the number of triangles attached to each vertex, 
i.e., the integer~$c_i$ is the number of triangles incident to the vertex $i$. One has $c_0=c_{k+1}=1$, and $c_i\geq 2$ otherwise.
\end{enumerate}

In addition the triangulation $\T$ of \eqref{trigen} comes with an orientation of the edges: 
\begin{itemize}
\item an edge adjacent to a triangle base up at its right and a triangle
 base down at its left is oriented upward;
\item an edge adjacent to a triangle base up at its left and a triangle
 base down at its right is oriented downward;
 \item an edge adjacent to two triangles base up  is oriented downward;
 \item an edge adjacent to two triangles
 base down is oriented upward;
 \item a base edge is oriented leftward;
  \item the rightmost edge is oriented downward;
 \item the leftmost edge has no orientation.
\end{itemize}

Note that we also have the following immediate relations:
$$
n:=\#\{ \text{vertices} \} = \#\{ \text{triangles} \} +2\;\;=\big(\sum_{1\leq i \leq 2m} a_i \big)+2\;\;=\sum _{1\leq i \leq k} (c_i-1)+3.
$$

\begin{rem}
\begin{enumerate}
\item The sequence $\left(c_0, c_1,\ldots, c_{n-1}\right)$ is called the \textit{quiddity} of the triangulated polygon in reference of the theory of Conway-Coxeter friezes \cite{CoCo}. The quiddity sequence determines a unique triangulation of a polygon. In our situation exactly two coefficients are equal to 1, namely $c_0=c_{k+1}=1$. In this case each subsequence $(c_1,\ldots,c_k)$ or $(c_{k+2},\ldots,c_{n-1})$ uniquely determines the triangulation of the polygon.

Alternatively the sequence
$(a_1,a_2,\ldots,a_{2m})$ also determines uniquely the triangulation. 

\item One has the relationship
\begin{equation*}
(c_1,\ldots,c_k)=
\big(a_1+2,\underbrace{2,\ldots,2}_{a_2-1},\,
a_3+2,\underbrace{2,\ldots,2}_{a_4-1},\ldots,
a_{2m-1}+2,\underbrace{2,\ldots,2}_{a_{2m}-1}\big)
 \end{equation*}
which is the exact conversion formula of Hirzebruch for the continued fractions expansions.
In other words, the sequences of integers $(c_1,\ldots,c_k)$ and $(a_1,a_2,\ldots,a_{2m})$  provided by  the triangulation lead to the same rational computed by the two types of continued fractions :
$$\frac{r}{s}=[a_{1}, \ldots, a_{2m}]_{}=\llbracket{}c_1,c_2,\ldots, c_{k}\rrbracket{}_{}.$$

\item The Farey triangulation $\T_{\frac{r}{s}}$ described in \S\ref{Trs} correspond to the triangulated polygon \eqref{trigen} with parameters $(c_1,\ldots,c_k)$ and $(a_1,a_2,\ldots,a_{2m})$, where these sequences respectively encode the negative and regular continued fraction expansions; see \cite{MGOfarey1} for details.

\end{enumerate}

\end{rem}
\subsection{Combinatorial interpretations of the $q$-continuants}
We reformulate the result of Theorem \ref{thar} in terms of $q$-continuants. We use the model of triangulated polygons instead of the Farey tessellation.

Starting from a sequence $(a_1,a_2,\ldots,a_{2m})$ of positive integers, or from a sequence $(c_1,\ldots,c_k)$ of integers greater than 1, denote by $\T$ the corresponding fan triangulated $n$-gon \eqref{trigen}. We consider paths that follow the oriented edges of $\T$. If $\pi $ is a path from vertex $k_1$ to vertex $k_2$ we write $\pi: k_1\to k_2$. 

In the triangulation $\T$, one defines the \textit{top path}
$$\tau:k+1\to k \to \cdots \to 2\to1$$
 and the \textit{bottom path}
 $$\beta: k+1\to k+3\to \cdots \to n-1 \to 1.$$ 
 The \textit{area} and \textit{coarea} of a path $\pi$ in  $\T$ are defined as
 \begin{eqnarray}\label{ardef2}
\ar(\pi)&:=&\#\{\text{triangles enclosed between $\pi$ and top path $\tau$}\}\\
\label{coardef2}
\coar(\pi)&:=&\#\{\text{triangles enclosed between $\pi$ and the bottom path of  $\beta$ }\}
\end{eqnarray}

Note that the rightmost triangle $t_0-\{0,1,n-1\}$ is never taken into account in the count for the area and the coarea.
The area, resp. coarea, corresponds to the number of triangles with three oriented edges above the path, resp. under the path. 

\begin{ex}
Starting from the sequence $(a_1, \ldots, a_{2m})=(1,2,1,1)$ or from $(c_1,\ldots, c_k)=(2,2,3)$ one obtains the following triangulated heptagon.
The top path is colored in blue and the bottom path in orange. 
\begin{center}
\begin{tikzpicture}[scale=0.5, baseline=(current  bounding  box.center) ]
		\draw (2,0)--(5, 0)--(4,4)--(2,0);
		\draw (4,4)--(8,4)--(5, 0);
		\draw (6,4)--(5,0); 
		\draw  (8,4)--(8,0)--(5,0);
		\draw (8,0)--(10,4)--(8,4);
		\draw [orange, very thick] (4,4)--(5,0)--(8,0)--(10,4);
		\draw [blue, very thick] (10,4)--(4,4);
		\draw [<-,red,thick] (3.5,0)-- (3.51,0);
		\draw [<-,red,thick] (6.5,0)-- (6.51,0);
		\draw [<-,red,thick] (7,4)-- (7.01,4);
		\draw [<-,red,thick] (5,4)-- (5.01,4);
		\draw [<-,red,thick] (9,4)-- (9.01,4);
		\draw [<-,red,thick] (9,2)-- (9.01,2.02);
		\draw [<-,red,thick] (6.5,2)-- (6.51,2.02);
		\draw [<-,red,thick] (5.5,2)-- (5.51,2.03);
		\draw [<-,red,thick] (4.4,2.3)-- (4.41,2.27);
		\draw [<-,red,thick] (8,2.3)-- (8.01,2.19);
		\draw (2,0) node[below] {$0$};
\draw (5,0) node[below] {$6$};
\draw (8,0) node[below] {$5$};
\draw (10,4) node[above] {$4$};
\draw (4,4) node[above] {$1$};
\draw (6,4) node[above] {$2$};
\draw (8,4) node[above] {$3$};
	\end{tikzpicture}
	\end{center}
	Examples of paths and of coareas of paths are given in Figure \ref{ex5euc} and \ref{ex7euc}. 
	Note that all the displayed examples in the fan triangulation model  correspond to Example \ref{triang75} and to Figures~\ref{ex5} and \ref{ex7} in the model of Farey tesselation.
\end{ex}

	\begin{figure}
	\begin{center}
	\begin{tikzpicture}[scale=0.4, baseline=(current  bounding  box.center) ]
		\draw (5.5,0) node[below] {$q^0$};
		\draw (4.9,0) node[below] {$ $};
		\draw (2,0)--(5, 0)--(4,4)--(2,0);
		\draw (4,4)--(8,4)--(5, 0);
		\draw (6,4)--(5,0); 
		\draw  (8,4)--(8,0)--(5,0);
		\draw (8,0)--(10,4)--(8,4);
		\draw [orange, very thick] (4,4)--(5,0)--(8,0)--(10,4);
		\draw [green, very thick] (10,4)--(8,0)--(2,0);
		\draw [very thin] (10,4)--(8,0)--(2,0);
		\draw [<-,red,thick] (3.5,0)-- (3.51,0);
		\draw [<-,red,thick] (6.5,0)-- (6.51,0);
		\draw [<-,red,thick] (7,4)-- (7.01,4);
		\draw [<-,red,thick] (5,4)-- (5.01,4);
		\draw [<-,red,thick] (9,4)-- (9.01,4);
		\draw [<-,red,thick] (9,2)-- (9.01,2.02);
		\draw [<-,red,thick] (6.5,2)-- (6.51,2.02);
		\draw [<-,red,thick] (5.5,2)-- (5.51,2.03);
		\draw [<-,red,thick] (4.4,2.3)-- (4.41,2.27);
		\draw [<-,red,thick] (8,2.3)-- (8.01,2.19);
	\end{tikzpicture}
	\begin{tikzpicture}[scale=0.4, baseline=(current  bounding  box.center) ]
		\draw (5.5,0) node[below] {$q^1$};
		\draw (4.9,0) node[below] {$ $};
		\draw (2,0)--(5, 0)--(4,4)--(2,0);
		\draw (4,4)--(8,4)--(5, 0);
		\draw (6,4)--(5,0); 
		\draw  (8,4)--(8,0)--(5,0);
		\draw (8,0)--(10,4)--(8,4);
		\draw [fill=gray!30] (5, 0)--(8,4)--(8,0)--(5,0);
		\draw [orange, very thick] (4,4)--(5,0)--(8,0)--(10,4);
		\draw [green, very thick] (10,4)--(8,0)--(8,4)--(5,0)--(2,0);
		\draw  [very thin] (10,4)--(8,0)--(8,4)--(5,0)--(2,0);
		\draw [<-,red,thick] (3.5,0)-- (3.51,0);
		\draw [<-,red,thick] (6.5,0)-- (6.51,0);
		\draw [<-,red,thick] (7,4)-- (7.01,4);
		\draw [<-,red,thick] (5,4)-- (5.01,4);
		\draw [<-,red,thick] (9,4)-- (9.01,4);
		\draw [<-,red,thick] (9,2)-- (9.01,2.02);
		\draw [<-,red,thick] (6.5,2)-- (6.51,2.02);
		\draw [<-,red,thick] (5.5,2)-- (5.51,2.03);
		\draw [<-,red,thick] (4.4,2.3)-- (4.41,2.27);
		\draw [<-,red,thick] (8,2.3)-- (8.01,2.19);
	\end{tikzpicture}	
	\begin{tikzpicture}[scale=0.4, baseline=(current  bounding  box.center) ]
		\draw (5.5,0) node[below] {$q^2$};
		\draw (4.9,0) node[below] {$ $};
		\draw [fill=gray!30] (5, 0)--(6,4)--(8,4)--(8,0)--(5,0);
		\draw (2,0)--(5, 0)--(4,4)--(2,0);
		\draw (4,4)--(8,4)--(5, 0);
		\draw (6,4)--(5,0); 
		\draw  (8,4)--(8,0)--(5,0);
		\draw (8,0)--(10,4)--(8,4);
		\draw [orange, very thick] (4,4)--(5,0)--(8,0)--(10,4);
		\draw [green, very thick] (10,4)--(8,0)--(8,4)--(6,4)--(5,0)--(2,0);
		\draw [very thin] (10,4)--(8,0)--(8,4)--(6,4)--(5,0)--(2,0);
		\draw [<-,red,thick] (3.5,0)-- (3.51,0);
		\draw [<-,red,thick] (6.5,0)-- (6.51,0);
		\draw [<-,red,thick] (7,4)-- (7.01,4);
		\draw [<-,red,thick] (5,4)-- (5.01,4);
		\draw [<-,red,thick] (9,4)-- (9.01,4);
		\draw [<-,red,thick] (9,2)-- (9.01,2.02);
		\draw [<-,red,thick] (6.5,2)-- (6.51,2.02);
		\draw [<-,red,thick] (5.5,2)-- (5.51,2.03);
		\draw [<-,red,thick] (4.4,2.3)-- (4.41,2.27);
		\draw [<-,red,thick] (8,2.3)-- (8.01,2.19);
	\end{tikzpicture}
	
	\begin{tikzpicture}[scale=0.4, baseline=(current  bounding  box.center) ]
		\draw (5.5,0) node[below] {$q^2$};
		\draw (4.9,0) node[below] {$ $};
		\draw [fill=gray!30] (5, 0)--(8,0)--(10,4)--(8,4)--(5,0);
		\draw (2,0)--(5, 0)--(4,4)--(2,0);
		\draw (4,4)--(8,4)--(5, 0);
		\draw (6,4)--(5,0); 
		\draw  (8,4)--(8,0)--(5,0);
		\draw (8,0)--(10,4)--(8,4);
		\draw [orange, very thick] (4,4)--(5,0)--(8,0)--(10,4);
		\draw [green, very thick] (10,4)--(8,4)--(5,0)--(2,0);
		\draw [very thin] (10,4)--(8,4)--(5,0)--(2,0);
		\draw [<-,red,thick] (3.5,0)-- (3.51,0);
		\draw [<-,red,thick] (6.5,0)-- (6.51,0);
		\draw [<-,red,thick] (7,4)-- (7.01,4);
		\draw [<-,red,thick] (5,4)-- (5.01,4);
		\draw [<-,red,thick] (9,4)-- (9.01,4);
		\draw [<-,red,thick] (9,2)-- (9.01,2.02);
		\draw [<-,red,thick] (6.5,2)-- (6.51,2.02);
		\draw [<-,red,thick] (5.5,2)-- (5.51,2.03);
		\draw [<-,red,thick] (4.4,2.3)-- (4.41,2.27);
		\draw [<-,red,thick] (8,2.3)-- (8.01,2.19);
	\end{tikzpicture}	
	\begin{tikzpicture}[scale=0.4, baseline=(current  bounding  box.center) ]
		\draw (5.5,0) node[below] {$q^3$};
		\draw (4.9,0) node[below] {$ $};
		\draw [fill=gray!30] (5, 0)--(6,4)--(10,4)--(8,0)--(5,0);
		\draw (2,0)--(5, 0)--(4,4)--(2,0);
		\draw (4,4)--(8,4)--(5, 0);
		\draw (6,4)--(5,0); 
		\draw  (8,4)--(8,0)--(5,0);
		\draw (8,0)--(10,4)--(8,4);
		\draw [orange, very thick] (4,4)--(5,0)--(8,0)--(10,4);
		\draw [green, very thick] (10,4)--(6,4)--(5,0)--(2,0);
		\draw [orange, very thick] (4,4)--(5,0)--(8,0)--(10,4);
		\draw [very thin] (10,4)--(6,4)--(5,0)--(2,0);
		\draw [<-,red,thick] (3.5,0)-- (3.51,0);
		\draw [<-,red,thick] (6.5,0)-- (6.51,0);
		\draw [<-,red,thick] (7,4)-- (7.01,4);
		\draw [<-,red,thick] (5,4)-- (5.01,4);
		\draw [<-,red,thick] (9,4)-- (9.01,4);
		\draw [<-,red,thick] (9,2)-- (9.01,2.02);
		\draw [<-,red,thick] (6.5,2)-- (6.51,2.02);
		\draw [<-,red,thick] (5.5,2)-- (5.51,2.03);
		\draw [<-,red,thick] (4.4,2.3)-- (4.41,2.27);
		\draw [<-,red,thick] (8,2.3)-- (8.01,2.19);
	\end{tikzpicture}
	\caption{The 5 paths starting at vertex $4$ and ending at vertex 0. The triangles shaded in gray correspond to the triangles contributing to the coarea of the path. The coarea-generating polynomial $1+q+2q^2+q^3$ corresponds to the denominator of~$\left[\frac75\right]_q$. }
	\label{ex5euc}
	\end{center}
	\end{figure}
	
	\begin{figure}[h!]
	\begin{center}
		\begin{tikzpicture}[scale=0.4]
		\draw (5.5,0) node[below] {$q^0$};
		\draw (2,0)--(5, 0)--(4,4)--(2,0);
		\draw (4,4)--(8,4)--(5, 0);
		\draw (6,4)--(5,0); 
		\draw  (8,4)--(8,0)--(5,0);
		\draw (8,0)--(10,4)--(8,4);
		\draw [orange, very thick] (4,4)--(5,0)--(8,0)--(10,4);
		\draw [green, very thick] (4,4)--(5,0)--(8,0)--(10,4);
		\draw [very thin]  (4,4)--(5,0)--(8,0)--(10,4);
		\draw [<-,red,thick] (3.5,0)-- (3.51,0);
		\draw [<-,red,thick] (6.5,0)-- (6.51,0);
		\draw [<-,red,thick] (7,4)-- (7.01,4);
		\draw [<-,red,thick] (5,4)-- (5.01,4);
		\draw [<-,red,thick] (9,4)-- (9.01,4);
		\draw [<-,red,thick] (9,2)-- (9.01,2.02);
		\draw [<-,red,thick] (6.5,2)-- (6.51,2.02);
		\draw [<-,red,thick] (5.5,2)-- (5.51,2.03);
		\draw [<-,red,thick] (4.4,2.3)-- (4.41,2.27);
		\draw [<-,red,thick] (8,2.3)-- (8.01,2.19);
	\end{tikzpicture}
		\begin{tikzpicture}[scale=0.4]
		\draw (5.5,0) node[below] {$q^1$};
		\draw [fill=gray!30] (5, 0)--(8,4)--(8,0)--(5,0);
		\draw (2,0)--(5, 0)--(4,4)--(2,0);
		\draw (4,4)--(8,4)--(5, 0);
		\draw (6,4)--(5,0); 
		\draw  (8,4)--(8,0)--(5,0);
		\draw (8,0)--(10,4)--(8,4);
		\draw [orange, very thick] (4,4)--(5,0)--(8,0)--(10,4);
		\draw [green, very thick] (10,4)--(8,0)--(8,4)--(5,0)--(4,4);
		\draw [green, very thick] (10,4)--(8,0)--(8,4)--(5,0)--(4,4);
		\draw [<-,red,thick] (3.5,0)-- (3.51,0);
		\draw [<-,red,thick] (6.5,0)-- (6.51,0);
		\draw [<-,red,thick] (7,4)-- (7.01,4);
		\draw [<-,red,thick] (5,4)-- (5.01,4);
		\draw [<-,red,thick] (9,4)-- (9.01,4);
		\draw [<-,red,thick] (9,2)-- (9.01,2.02);
		\draw [<-,red,thick] (6.5,2)-- (6.51,2.02);
		\draw [<-,red,thick] (5.5,2)-- (5.51,2.03);
		\draw [<-,red,thick] (4.4,2.3)-- (4.41,2.27);
		\draw [<-,red,thick] (8,2.3)-- (8.01,2.19);
	\end{tikzpicture}
		\begin{tikzpicture}[scale=0.4]
		\draw (5.5,0) node[below] {$q^2$};
		\draw [fill=gray!30] (5, 0)--(8,0)--(10,4)--(8,4)--(5,0);
		\draw (2,0)--(5, 0)--(4,4)--(2,0);
		\draw (4,4)--(8,4)--(5, 0);
		\draw (6,4)--(5,0); 
		\draw  (8,4)--(8,0)--(5,0);
		\draw (8,0)--(10,4)--(8,4);
		\draw [orange, very thick] (4,4)--(5,0)--(8,0)--(10,4);
		\draw [green, very thick] (10,4)--(8,4)--(5,0)--(4,4);
		\draw [very thin]  (10,4)--(8,4)--(5,0)--(4,4);
		\draw [<-,red,thick] (3.5,0)-- (3.51,0);
		\draw [<-,red,thick] (6.5,0)-- (6.51,0);
		\draw [<-,red,thick] (7,4)-- (7.01,4);
		\draw [<-,red,thick] (5,4)-- (5.01,4);
		\draw [<-,red,thick] (9,4)-- (9.01,4);
		\draw [<-,red,thick] (9,2)-- (9.01,2.02);
		\draw [<-,red,thick] (6.5,2)-- (6.51,2.02);
		\draw [<-,red,thick] (5.5,2)-- (5.51,2.03);
		\draw [<-,red,thick] (4.4,2.3)-- (4.41,2.27);
		\draw [<-,red,thick] (8,2.3)-- (8.01,2.19);
	\end{tikzpicture}
		\begin{tikzpicture}[scale=0.4]
		\draw (5.5,0) node[below] {$q^2$};
		\draw [fill=gray!30] (5, 0)--(8,0)--(8,4)--(6,4)--(5,0);
		\draw (2,0)--(5, 0)--(4,4)--(2,0);
		\draw (4,4)--(8,4)--(5, 0);
		\draw (6,4)--(5,0); 
		\draw  (8,4)--(8,0)--(5,0);
		\draw (8,0)--(10,4)--(8,4);
		\draw [orange, very thick] (4,4)--(5,0)--(8,0)--(10,4);
		\draw [green, very thick] (10,4)--(8,0)--(8,4)--(6,4)--(5,0)--(4,4);
		\draw [very thin]  (10,4)--(8,0)--(8,4)--(6,4)--(5,0)--(4,4);
		\draw [<-,red,thick] (3.5,0)-- (3.51,0);
		\draw [<-,red,thick] (6.5,0)-- (6.51,0);
		\draw [<-,red,thick] (7,4)-- (7.01,4);
		\draw [<-,red,thick] (5,4)-- (5.01,4);
		\draw [<-,red,thick] (9,4)-- (9.01,4);
		\draw [<-,red,thick] (9,2)-- (9.01,2.02);
		\draw [<-,red,thick] (6.5,2)-- (6.51,2.02);
		\draw [<-,red,thick] (5.5,2)-- (5.51,2.03);
		\draw [<-,red,thick] (4.4,2.3)-- (4.41,2.27);
		\draw [<-,red,thick] (8,2.3)-- (8.01,2.19);
	\end{tikzpicture}

		\begin{tikzpicture}[scale=0.4, baseline=(current  bounding  box.center) ]
		\draw (5.5,0) node[below] {$q^3$};
		\draw [fill=gray!30] (5, 0)--(8,0)--(10,4)--(6,4)--(5,0);
		\draw (2,0)--(5, 0)--(4,4)--(2,0);
		\draw (4,4)--(8,4)--(5, 0);
		\draw (6,4)--(5,0); 
		\draw  (8,4)--(8,0)--(5,0);
		\draw (8,0)--(10,4)--(8,4);
		\draw [orange, very thick] (4,4)--(5,0)--(8,0)--(10,4);
		\draw [green, very thick] (10,4)--(6,4)--(5,0)--(4,4);
		\draw [green, very thick] (10,4)--(6,4)--(5,0)--(4,4);
		\draw [<-,red,thick] (3.5,0)-- (3.51,0);
		\draw [<-,red,thick] (6.5,0)-- (6.51,0);
		\draw [<-,red,thick] (7,4)-- (7.01,4);
		\draw [<-,red,thick] (5,4)-- (5.01,4);
		\draw [<-,red,thick] (9,4)-- (9.01,4);
		\draw [<-,red,thick] (9,2)-- (9.01,2.02);
		\draw [<-,red,thick] (6.5,2)-- (6.51,2.02);
		\draw [<-,red,thick] (5.5,2)-- (5.51,2.03);
		\draw [<-,red,thick] (4.4,2.3)-- (4.41,2.27);
		\draw [<-,red,thick] (8,2.3)-- (8.01,2.19);
	\end{tikzpicture}
		\begin{tikzpicture}[scale=0.4, baseline=(current  bounding  box.center) ]
		\draw (5.5,0) node[below] {$q^3$};
		\draw [fill=gray!30] (5, 0)--(8,0)--(8,4)--(4,4)--(5,0);
		\draw (2,0)--(5, 0)--(4,4)--(2,0);
		\draw (4,4)--(8,4)--(5, 0);
		\draw (6,4)--(5,0); 
		\draw  (8,4)--(8,0)--(5,0);
		\draw (8,0)--(10,4)--(8,4);
		\draw [orange, very thick] (4,4)--(5,0)--(8,0)--(10,4);
		\draw [green, very thick] (10,4)--(8,0)--(8,4)--(4,4);
		\draw [very thin]  (10,4)--(8,0)--(8,4)--(4,4);
		\draw [<-,red,thick] (3.5,0)-- (3.51,0);
		\draw [<-,red,thick] (6.5,0)-- (6.51,0);
		\draw [<-,red,thick] (7,4)-- (7.01,4);
		\draw [<-,red,thick] (5,4)-- (5.01,4);
		\draw [<-,red,thick] (9,4)-- (9.01,4);
		\draw [<-,red,thick] (9,2)-- (9.01,2.02);
		\draw [<-,red,thick] (6.5,2)-- (6.51,2.02);
		\draw [<-,red,thick] (5.5,2)-- (5.51,2.03);
		\draw [<-,red,thick] (4.4,2.3)-- (4.41,2.27);
		\draw [<-,red,thick] (8,2.3)-- (8.01,2.19);
	\end{tikzpicture}
		\begin{tikzpicture}[scale=0.4, baseline=(current  bounding  box.center) ]
		\draw (5.5,0) node[below] {$q^4$};
		\draw [fill=gray!30] (5, 0)--(8,0)--(10,4)--(4,4)--(5,0);
		\draw (2,0)--(5, 0)--(4,4)--(2,0);
		\draw (4,4)--(8,4)--(5, 0);
		\draw (6,4)--(5,0); 
		\draw  (8,4)--(8,0)--(5,0);
		\draw (8,0)--(10,4)--(8,4);
		\draw [orange, very thick] (4,4)--(5,0)--(8,0)--(10,4);
		\draw [green, very thick] (10,4)--(4,4);
		\draw [very thin]  (10,4)--(4,4);
		\draw [<-,red,thick] (3.5,0)-- (3.51,0);
		\draw [<-,red,thick] (6.5,0)-- (6.51,0);
		\draw [<-,red,thick] (7,4)-- (7.01,4);
		\draw [<-,red,thick] (5,4)-- (5.01,4);
		\draw [<-,red,thick] (9,4)-- (9.01,4);
		\draw [<-,red,thick] (9,2)-- (9.01,2.02);
		\draw [<-,red,thick] (6.5,2)-- (6.51,2.02);
		\draw [<-,red,thick] (5.5,2)-- (5.51,2.03);
		\draw [<-,red,thick] (4.4,2.3)-- (4.41,2.27);
		\draw [<-,red,thick] (8,2.3)-- (8.01,2.19);
	\end{tikzpicture}
	\caption{The 7 paths starting at vertex $4$ and ending at vertex 1. The triangles shaded in gray correspond to the triangles contributing to the coarea of the path. The coarea-generating polynomial $1+q+2q^2+2q^3+q^4$ corresponds to the numerator of~$\left[\frac75\right]_q$. }
	\label{ex7euc}
\end{center}
	\end{figure}

\begin{prop}\label{intcont}
With the above notation, one has
\begin{equation*}
\begin{array}{lclcl}
K^{}_{2m}(a_1,\ldots,a_{2m})_{q}&=&E_k(c_1,\ldots,c_k)_q&=&\displaystyle\sum_{\pi\,:\,k+1\to1}q^{\coar(\pi)}\;,\\[20pt]
K^{}_{2m-1}(a_2,\ldots,a_{2m})_{q}&=&E_{k-1}(c_2,\ldots,c_k)_q&=&\displaystyle\sum_{\pi\,:\,k+1\to0}q^{\coar(\pi)},
\end{array}
\end{equation*}
where the sums run over all the paths $\pi$ in $\T$ starting at vertex $k+1$ and ending at vertex 1 or vertex 0 respectively.
\end{prop}

\begin{proof} Consider the corresponding rational number defined by
$$\frac{r}{s}:=[a_{1}, \ldots, a_{2m}]_{}=\llbracket{}c_1,c_2,\ldots, c_{k}\rrbracket{}_{}.$$
As explained in \S\ref{Trs} (see also \cite{MGOfarey1} for more details) the Farey triangulation $\T_{\frac{r}{s}}$ contains the sequence $(a_{1}, \ldots, a_{2m})$ in the distribution of triangles according to the position of the bases at the left or right of the vertical line drawn from $\frac{r}{s}$. The triangulation $\T$ with parameters $(a_{1}, \ldots, a_{2m})$ is a redrawing of  $\T_{\frac{r}{s}}$ using Euclidean triangles. The vertical line can be thought in $\T$ as the diagonal joining the vertices $0$ and $k+1$ and the position left or right of the base of a Farey triangle becomes position under or above this diagonal.
Since
$$\dfrac{\Rc}{\Sc}=\left[\frac{r}{s}\right]_{q}=\dfrac{K_{n}(a_1,\ldots,a_n)_q}{K_{n-1}(a_2,\ldots,a_{n})_q}
         =\dfrac{E_{k}(c_1,\ldots,c_k)_q}{E_{k-1}(c_2,\ldots,c_{k})_q}
$$
the proposition is a simple reformulation of Theorem \ref{thar}.
\end{proof}

\begin{rem}
When $q=1$ the proposition states that the numerator $r=K^{}_{2m}(a_1,\ldots,a_{2m})$ is the total number of paths from  vertex $k+1$ to vertex 1 and the denominator $s=K^{}_{2m-1}(a_2,\ldots,a_{2m})$ is the total number of paths from  vertex $k+1$ to vertex 0. This result can be found in \cite[Prop. 5.21]{PoP} in the language of lotuses. It is also equivalent to the results of \cite{Pro} and \cite{CaSc} in terms of paths or matchings in  snake graphs.
\end{rem}

\begin{rem}\label{intcont2} The Farey triangulation $\T_\frac{r}{s}$ contains the two 
rationals defined by $\widetilde{\frac{r}{s}}:=[a_{1}, \ldots, a_{2m-1}]_{}$ and $\overline{\frac{r}{s}}:=\llbracket{}c_1,c_2,\ldots, c_{k-1}\rrbracket{}_{}$. These rationals are called \textit{convergents} of  $\frac{r}{s}=[a_{1}, \ldots, a_{2m}]_{}=\llbracket{}c_1,c_2,\ldots, c_{k}\rrbracket{}_{}$. The convergents appear in the second column of the matrices \eqref{MatEq}. In the fan triangulation $\T$ these convergents correspond to vertex $k+2$ and vertex $k$ respectively. The fan triangulations corresponding to the convergents are included in $\T$. They are obtained by removing all the triangles but one incident to the vertex $k$ or $k+2$.
Therefore changing the initial vertex in the paths we obtain similar formula for the convergents: 
\begin{equation*}
\begin{array}{lcllcl}
\widetilde{K}^{}_{2m-1}(a_1,\ldots,a_{2m-1})_{q}&=&\displaystyle\sum_{\pi\,:\,k+2\to1}q^{\coar(\pi)}\;,&
\widetilde{K}^{}_{2m-2}(a_2,\ldots,a_{2m-1})_{q}&=&\displaystyle\sum_{\pi\,:\,k+2\to0}q^{\coar(\pi)},\\[20pt]
E_{k-1}(c_1,\ldots,c_{k-1})_q&=&\displaystyle\sum_{\pi\,:\,k\to1}q^{\coar(\pi)}\;,&
E_{k-2}(c_2,\ldots,c_{k-1})_q&=&\displaystyle\sum_{\pi\,:\,k\to1}q^{\coar(\pi)}\;,\\[10pt]
\end{array}
\end{equation*}
where the sums run over paths $\pi$ in $\T$.
\end{rem}

\section{$q$-rotundus and triangulations of annuli}\label{rotund}
In this section we define the quantum rotundi and give interpretations using triangulation of annuli.
\subsection{Definition}
The quantum rotundi are the $q$-analogues of \eqref{RotsEq} defined by
\begin{equation}
\label{qRotsEq}
\begin{array}{lclcl}
R^+ (a_1,\ldots,a_n)_q&:=&qK_n (a_1,\ldots,a_n)_q+
\widetilde{K}_{n-2} (a_2,\ldots,a_{n-1})_q&=&\Tr M^+ (a_1,\ldots,a_n)_q
\\[12pt]
R(c_1,\ldots,c_k)_q&:=&E_{k}(c_1,\ldots,c_k)_q-q^{c_{k}-1}E_{k-2}(c_2,\ldots,c_{k-1})_q &=&\Tr M (c_1,\ldots,c_k)_q\;,
\end{array}
\end{equation}
where $n$ is even and  $a_i$ and $c_i$ are positive integers. The $q$-rotundi are polynomials in $q$.
Traces of $q$-deformed matrices of $\SL(2,\Z)$ were studied in \cite{LMGadv} and in particular one has the following property.
\begin{thm}[\cite{LMGadv}]\label{palin}
The $q$-rotundi are palindromic polynomials in $q$ with positive integer coefficients.
\end{thm}
The above theorem is based on the equalities obtained by reversing the sequences of parameters
$$
\Tr M (c_1,\ldots,c_k)_q=\Tr M (c_k,\ldots,c_1)_q, \text{ and }\; \Tr M^+ (a_1,\ldots,a_n)_q =\Tr M^+ (a_n,\ldots,a_1)_q
$$
which are not as immediate to establish as in the case $q=1$, see \cite[Lemma 3.8]{LMGadv}.
\begin{ex} The rotundi associated with the continued fraction expansions of $\frac75$ are
$$
\begin{array}{lclcl}
R^+(1,2,1,1)_q&:=&\Tr 
\begin{pmatrix}
q+q^2+2q^3+2q^4+q^5&1+q+q^2+q^3\\[6pt]
q+q^2+2q^3+q^4&1+q+q^2
\end{pmatrix}\\[16pt]&=&1+2q+2q^2+2q^3+2q^4+q^5
\\[20pt]
R (2,2,3)_q&:=&\Tr 
\begin{pmatrix}
1+q+2q^2+2q^3+q^4&-(q^2+q^3+q^4)\\[6pt]
1+q+2q^2+q^3&-(q^2+q^3)
\end{pmatrix}\\[16pt]&=&1+q+q^2+q^3+q^4.
\end{array}
$$
\end{ex}

\subsection{Triangulations of annuli}
We use the terminology of \cite{FST} for general triangulated surfaces, and the notation of \cite{BaAl2} in the case of triangulated annuli.
An annulus is the region bounded by two concentric circles. We denote by $C_{\ell,k}$ an annulus with $\ell$ marked points on the outer circle and $k$ marked points on the inner circle. The marked points may be connected with (oriented) arcs. There are \textit{boundary arcs} which connect two consecutive marked points along the boundary circles. There are \textit{bridging arcs} which connect two marked points on different boundary circles. And there are \textit{peripheral arcs} which connect two marked points on the same boundary. 
 A triangulation of an annulus is a maximal collection of arcs (up to homotopy) that do not intersect in the interior of the annulus.
 Boundary arcs always belong to a triangulation. In our situation, the triangulations will not involve peripheral arcs and all arcs will be oriented. A triangle is a closed region in the annulus bounded by three connected arcs.
 
We define triangulations of annuli associated to the sequences $(a_1,a_2,\ldots,a_{2m})$  of positive integers and $(c_1,\ldots,c_k)$ of integers greater than 1.
 
 \begin{itemize}
\item Let  $\T^-(c_1,\ldots,c_k)$ be the oriented triangulation of $C_{n-k-3,k}$ obtained from the fan triangulation $\T$ by gluing the triangle $\{0,1, n-1\}$ with the triangle $\{k, k+1, k+2\}$, the latter imposes the orientation. More precisely, vertices $0,1, n-1$ are respectively glued on vertices $k, k+1, k+2$. In the resulting triangulated annulus, the inner boundary has marked points numbered from $1$ to $k$ anticlockwise, and the outer boundary has marked points numbered from $k+2$ to $n-1$ clockwise.
\item Let $\T^+(a_1,a_2,\ldots,a_{2m})$ be the oriented triangulation of $C_{n-k-2,k}$ obtained from the fan triangulation $\T$ by gluing  the edge joining $1$ and $0$ with the edge joining $k+1$ and $k+2$ so that 1 is glued on $k+1$ and $0$ on $k+2$ without flipping.  The inner boundary has marked points labeled $1, \ldots, k$ anticlockwise, and the outer boundary has marked points labeled $k+2, \ldots, n$ clockwise.
\end{itemize}

 \begin{ex}  For the sequences $(a_1,a_2,\ldots,a_{2m})=(1,2,1,1)$ and $(c_1,\ldots,c_k)=(2,2,3)$ one obtains the following triangulations of annuli:
\begin{center}
\includegraphics[width=10cm]{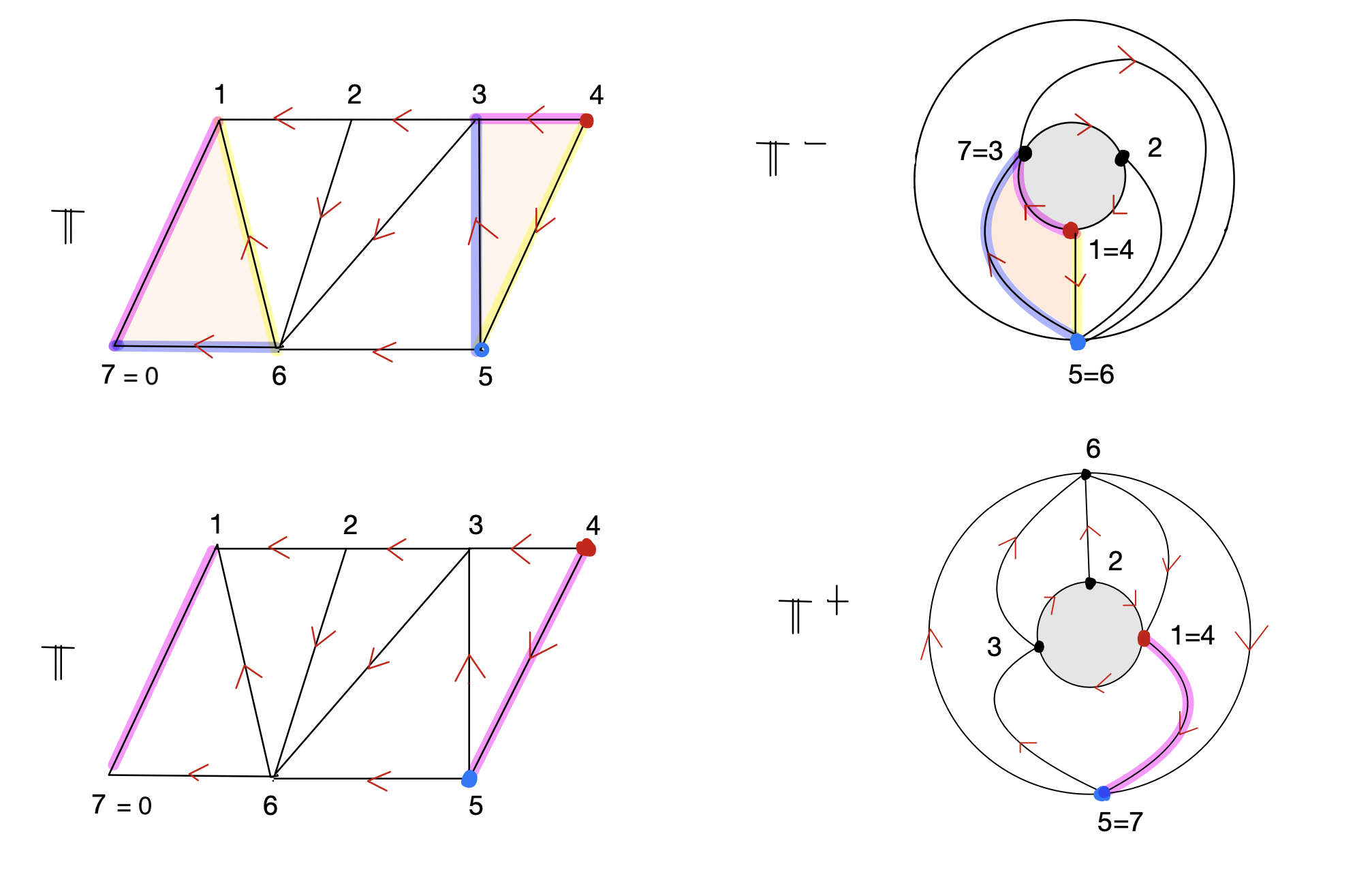}
\end{center}
\end{ex}

 \begin{rem}\label{t+t-}
In  $\T^{-}(c_1,\ldots,c_k)$  one recovers the sequence $(c_1,\ldots,c_k)$, up to cyclic permutation, as the quiddity sequence attached to the points of the inner boundary. In $\T^{+}(a_1,a_2,\ldots,a_{2m})$  one recovers the sequence  $(a_1,a_2,\ldots,a_{2m})$ by the alternating sequences of consecutive triangles with base on the outer/inner boundaries.
One has the relation $\T^{+}(a_1,a_2,\ldots,a_{2m})=\T^{-}(c_1+1,c_2, \ldots,c_k)$.
\end{rem}

 In the triangulations we consider oriented closed loops with no self-crossing. They are loops obtained by concatenation of oriented connected arcs given in the triangulation. For a closed loop $\gamma$ in the triangulated annulus we define the \textit{area} and the \textit{coarea} by
  \begin{eqnarray}\label{ardef3}
\ar(\gamma)&:=&\#\{\text{triangles enclosed between $\gamma$ and the inner boundary}\}\\
\label{coardef3}
\coar(\gamma)&:=&\#\{\text{triangles enclosed between $\gamma$ and the outer boundary }\}
\end{eqnarray}
 
 We are now ready to state the main result concerning enumerative interpretations of the $q$-rotundi. 
 The proof is postponed to section \ref{pf2}.
 \begin{thm}\label{thmrot}
Let $(a_1,a_2,\ldots,a_{2m})$ be a sequence of positive integers and let $(c_1,\ldots,c_k)$ be a sequence of integers greater than 1. One has
\begin{equation*}
\begin{array}{lclclcl}
(i) && R^{+}_{2m}(a_1,\ldots,a_{2m})_{q}&=&\displaystyle\sum_{\gamma \text{ in } \T^{+}}q^{\ar(\gamma)}
&=&\displaystyle\sum_{\gamma \text{ in } \T^{+}}q^{\coar(\gamma)}\;,\\[20pt]
(ii) && R_{k}(c_1,\ldots,c_k)_q&=&\displaystyle\sum_{\gamma \text{ in } \T^{-}}q^{\ar(\gamma)}
&=&\displaystyle\sum_{\gamma \text{ in } \T^{-}}q^{\coar(\gamma)}\;,\\
\end{array}
\end{equation*}
where the sums run over all oriented closed loops in the triangulations $\T^+=\T^{+}(a_1,a_2,\ldots,a_{2m})$ and $\T^-=\T^{-}(c_1,\ldots,c_k)$, respectively.
 \end{thm}
 Note that the sums involving the area and coarea coincide due to the palindromicity property mentioned in Theorem \ref{palin}.
 
 In the case $q=1$ one immediately gets the following corollary.
 
 \begin{cor}\label{corot}
 Let $(a_1,a_2,\ldots,a_{2m})$ be a sequence of positive integers and let $(c_1,\ldots,c_k)$ be a sequence of integers greater than 1.
 The rotundus $R^{+}_{2m}(a_1,\ldots,a_{2m})$ is the total number of closed loops in $\T^{+}(a_1,a_2,\ldots,a_{2m})$ and the rotundus
 $R_{k}(c_1,\ldots,c_k)$ is the total number of closed loops in $\T^{-}(c_1,\ldots,c_k)$.
 \end{cor}
 \begin{ex}
Going back to the example of $(a_1,a_2,\ldots,a_{2m})=(1,2,1,1)$ and $(c_1,\ldots,c_k)=(2,2,3)$, we obtain 10 closed loops in $\T^+$:

\includegraphics[width=15cm]{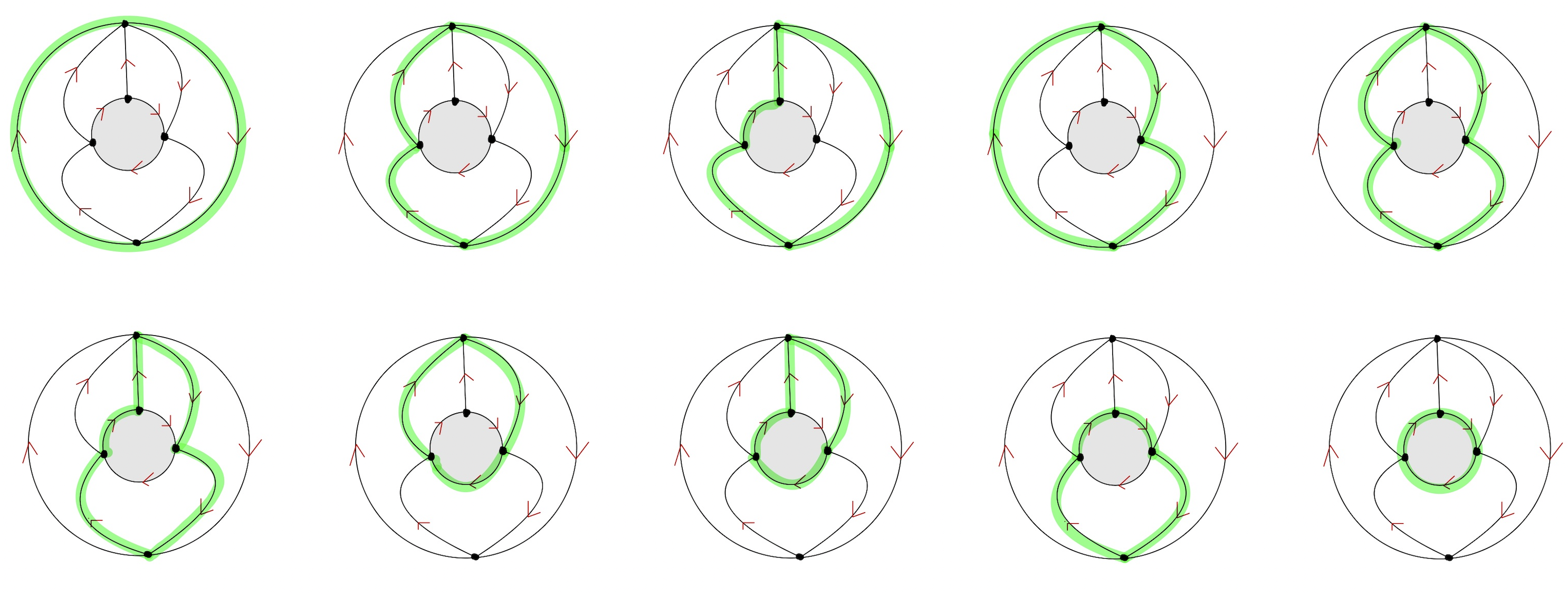}

The generating function for the area (or coarea) of these loops is $1+2q+2q^2+2q^3+2q^4+q^5$ which coincides with the $q$-rotundus $R^+_n (1,2,1,1)_q$.

We obtain 5 closed loops in $\T^-$:
\begin{center}
\includegraphics[width=15cm]{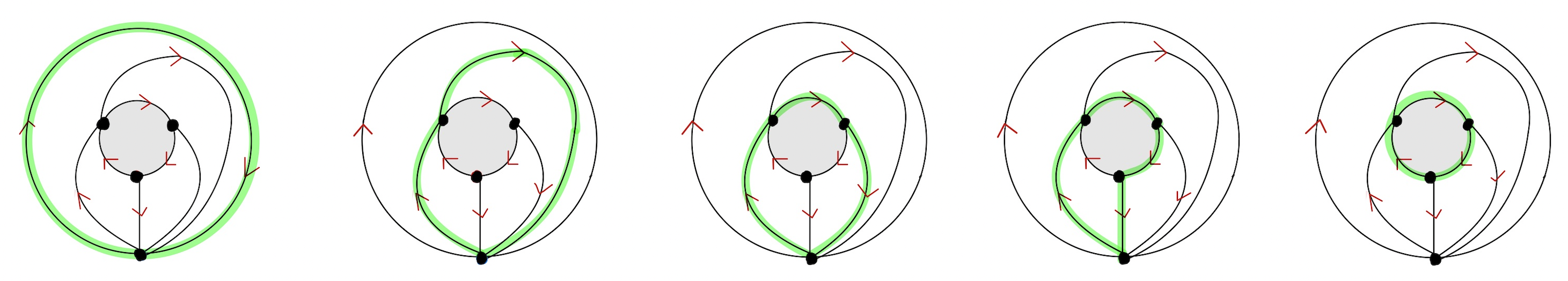}
\end{center}
The generating function for the area (or coarea) of theses loops is $1+q+q^2+q^3+q^4$ which coincides with the $q$-rotundus $R_k(2,2,3)_q$.
\end{ex}

\subsection{Proof of Theorem \ref{thmrot}}\label{pf2}
By Theorem \ref{palin} a formula with the area function is equivalent to the formula with the coarea function. We will use the coarea.

Let us start by proving part (i) of the theorem. By definition and by  Proposition \ref{intcont} and Remark \ref{intcont2} one has
\begin{equation*}
\begin{array}{lclclcl}
R^+_{2m}(a_1,\ldots,a_{2m})_q&=&qK_{2m}(a_1,\ldots,a_{2m})_q&+&
\widetilde{K}_{2m-2} (a_2,\ldots,a_{2m-1})_q\\[8pt]
&=&
\displaystyle q\sum_{\pi\,:\,k+1\to1}q^{\coar(\pi)}&+&\displaystyle\sum_{\pi\,:\,k+2\to0}q^{\coar(\pi)},
\end{array}
\end{equation*}
where the paths $\pi$ lie in the fan triangulation $\T$ associated with the sequences $(a_1,\ldots,a_{2m})$.

Each oriented closed  loop $\gamma$ in $\T^{+}$ gives rise to a path $\pi_\gamma$ in the fan triangulation $\T$, see Figure \ref{pathcurveex} at the end of the proof for an illustration. 
The coareas $\coar(\gamma)$ and $\coar(\pi_\gamma)$ either agree or differ by one. The triangle $t_0$ over the vertices $\{0, 1, n-1\}$ is never enclosed by a path in $\T$ but in can be enclosed or not by a loop in $\T^{+}$. There are two types of oriented closed loops:
\begin{small}
$$
\begin{array}{cccc}
\{\text{ loops in } \T^+\}=&
\{\text{ loops passing through vertex 1 } \} &\sqcup &
\left\{
\begin{array}{cc}
\text{ loops  passing through vertex  } k+2\\
\text{ and not passing through vertex 1} 
\end{array}
\right\}\\[8pt]
&\updownarrow&&\updownarrow\\[8pt]
&\{\text{ paths } k+1\to 1 \text{ in } \T\}
&&\{\text{ paths } k+2\to 0 \text{ in } \T\}
\end{array}
$$
\end{small}
If $\gamma$ passes through vertex 1, then the corresponding path $\pi_\gamma$ in $\T$ goes from vertex $k+1$ to vertex 1. The loop $\gamma$ as well as path $\pi_\gamma$ do not use the boundary arc or edge connecting vertices $n-1$ and $0$. All the triangles contributing in $\coar(\pi_\gamma)$ will contribute in $\coar(\gamma)$ but in addition the triangle $t_0$ will also contributes in $\coar(\gamma)$ . Hence, one has 
$$
q^{\coar(\gamma)}=q^{\coar(\pi_\gamma)+1},
$$
for all loops $\gamma$ passing through vertex 1.

If $\gamma$ does not passes through the vertex $1$ then it necessarily passes through the vertex $k+2$. The corresponding path $\pi_\gamma$ in $\T$ goes from vertex $k+2$ to vertex 0. The triangle $t_0$ will be enclosed by the loop $\gamma$ and will not contribute in $\coar(\gamma)$.
Hence, one has 
$$
q^{\coar(\gamma)}=q^{\coar(\pi_\gamma)},
$$
for all loops $\gamma$ not passing through vertex 1.

Finally we deduce 
\begin{eqnarray*}
\displaystyle\sum_{\gamma \text{ in } \T^{+}}q^{\coar(\gamma)}&=&
\displaystyle\sum_{\gamma \text{ through } 1}q^{\coar(\gamma)}+\displaystyle\sum_{\gamma\text{ not through }1}q^{\coar(\gamma)}\\[8pt]
&=& \displaystyle q\sum_{\pi\,:\,k+1\to1} q^{\coar(\pi)}+\displaystyle\sum_{\pi\,:\,k+2\to0}q^{\coar(\pi)}\\[8pt]
&=&R^+_{2m}(a_1,\ldots,a_{2m})_q
\end{eqnarray*}
Part (i) is proved.

To prove Part (ii) we use the following matrix relation taken from \cite[Prop 4.9]{MGOfmsigma}
$$
M^+_{2m}(a_1,\ldots,a_{2m})_q=M_k (c_1,\ldots,c_k)_qR_q,
$$
where $R_q=\begin{pmatrix} q&1\\0&1\end{pmatrix}$.
Taking the traces one gets
$$
\Tr M^+_{2m}(a_1,\ldots,a_{2m})_q=\Tr R_qM_k (c_1,\ldots,c_k)_q= \Tr M_k (c_1+1,c_2,\ldots,c_k)_q,
$$
and the result follows from $\T^{+}(a_1,a_2,\ldots,a_{2m})=\T^{-}(c_1+1,c_2,\ldots,c_k)$, see Remark \ref{t+t-}. Theorem \ref{thmrot} is proved.

\begin{figure}[h!]
\begin{center}
\includegraphics[width=12cm]{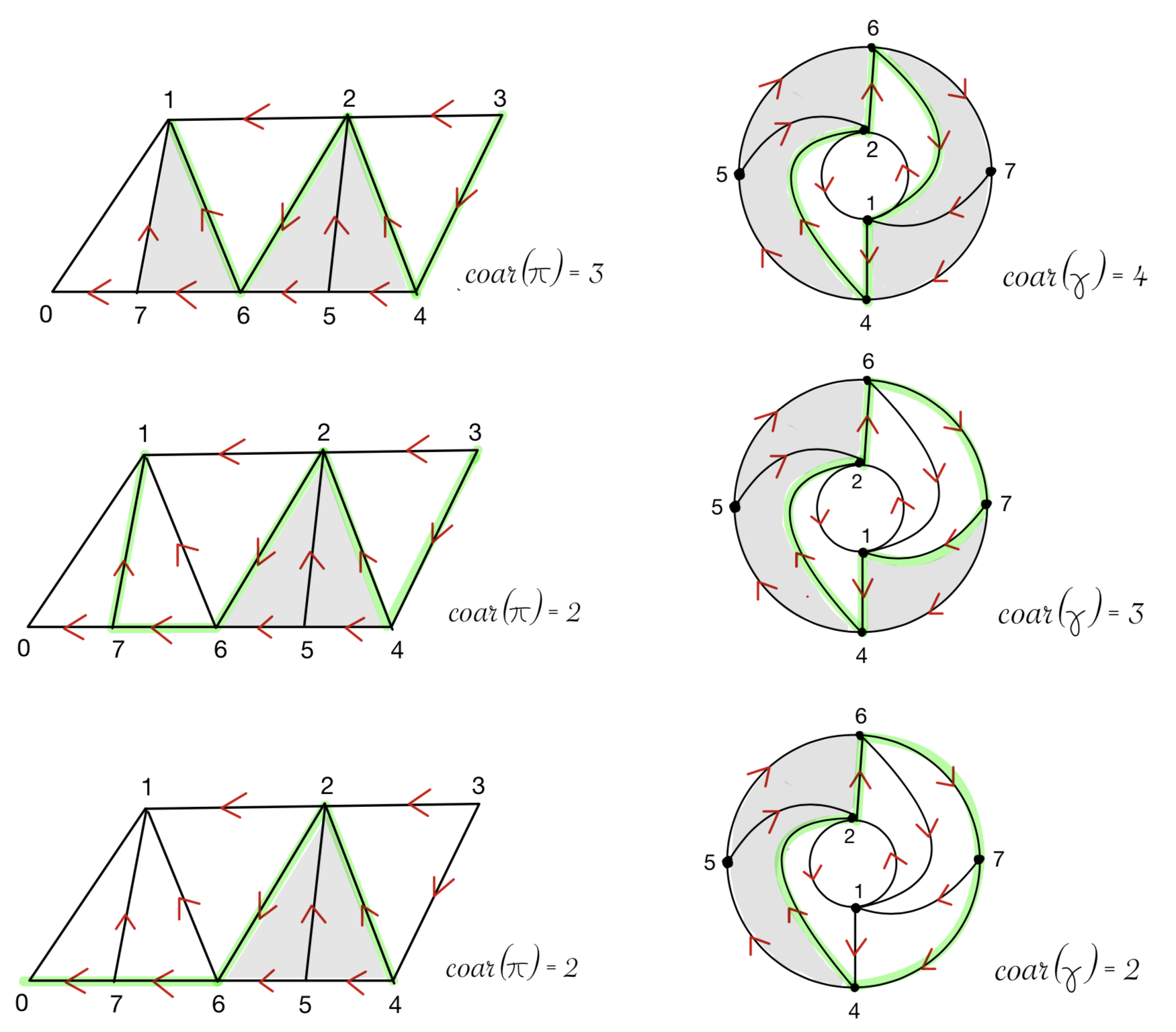}
\end{center}
\caption{Examples of closed loops in a triangulated annulus $\T^+$ with the corresponding paths in the associated triangulated polygon $\T$. The coarea of the curve and of the path are shaded in gray.}\label{pathcurveex}
\end{figure}

\section{Miscellaneous }\label{fin}

In this section we give extra formulas for the rotundi related to other combinatorial models or generalizing previous results.
\subsection{Matchings}
In the triangulated annulus $\T^-(c_1,\ldots, c_k)$ the vertices on the inner boundary are numbered from $1$ to $k$. A \textit{matching} in $\T^-(c_1,\ldots, c_k)$ is a $k$-tuple $(t_1, t_2, \ldots, t_k)$ of distinct triangles in $\T^-(c_1,\ldots, c_k)$ such that the triangle $t_i$ is incident to vertex $i$.
In \cite{BCI} a formula for the continuant $E_k(c_1, \ldots,c_k)$ is given in terms of matchings, see also \cite{BPT}. This formula implies the following formula for the rotundus
$$R_k(c_1, \ldots,c_k)=\#\{ \text{matchings in } \T^-(c_1,\ldots, c_k)\}.$$
This result does not involve the orientation of the triangulation unlike the result of Corollary~\ref{corot}.
\begin{ex}
In $\T^-(2,2,3)$ there are three points in the inner boundary and four triangles denoted $a, b, c, d$ as in the picture. One finds 5 matchings. This number coincides with $R(2,2,3)=5$.
$$
\begin{array}{cc}
\begin{array}{cc}
\includegraphics[width=3.4cm]{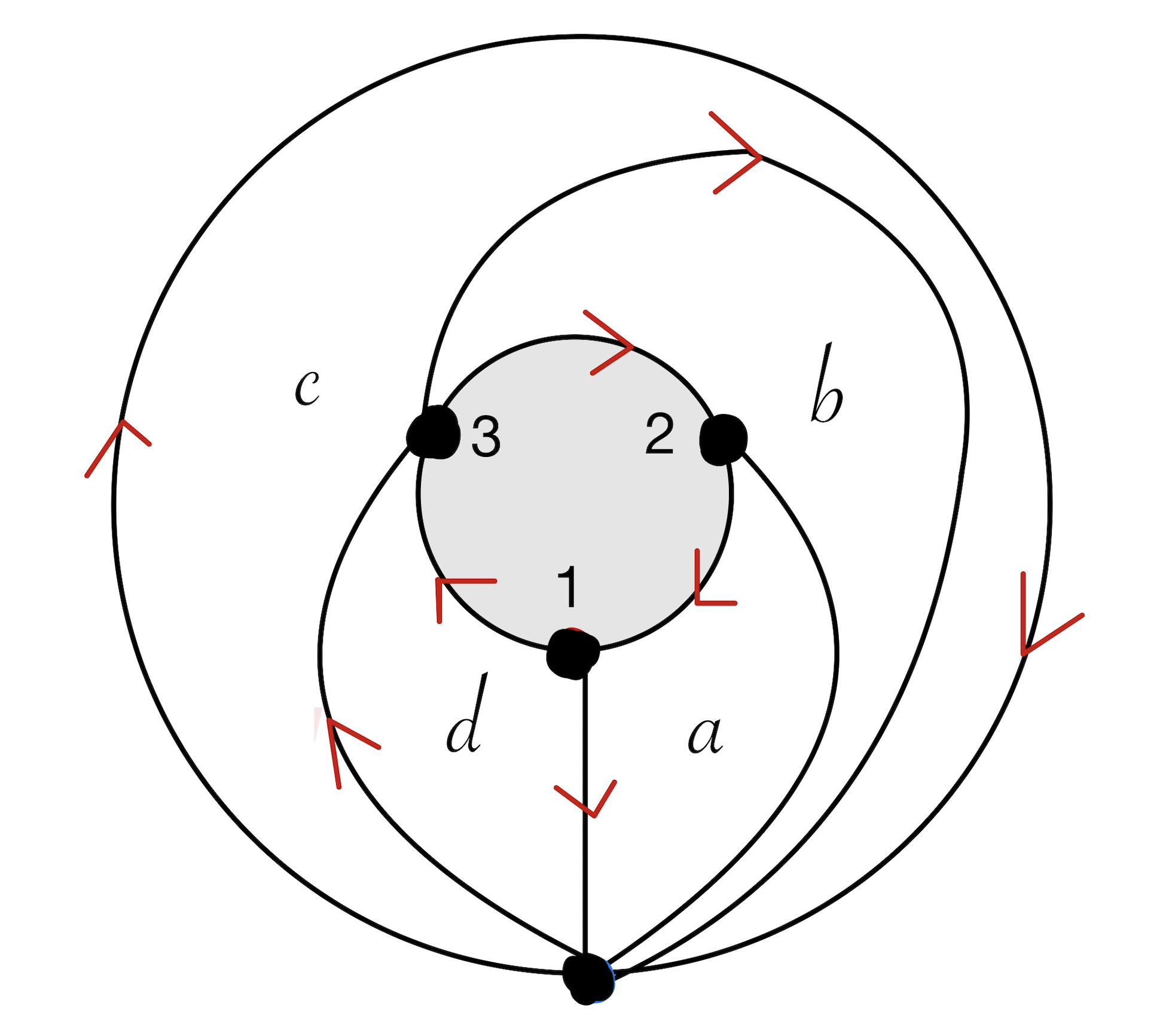}
\end{array}&
\begin{array}{cc}
&\text{matchings: }\\[6pt]
&(a,b,c)\\
&(a,b,d)\\
&(d,a,b)\\
&(d,a,c)\\
&(d,b,c)
\end{array}
\end{array}
$$
\end{ex}

\subsection{Dual graphs}
The dual graph associated 
with a triangulated polygon or triangulated annulus is defined in the following way. Each triangle is represented by a vertex and two vertices are linked by an edge if the corresponding triangles are adjacent. In our situation the dual graphs are oriented according to the orientation in the triangulations. The cyclic graph is not a full cycle, it has at least one source and at least one sink. 
Cyclic graphs from quiddity sequences have already appear in \cite[\S 3.2]{BaAl2}.

For instance, in the case of $\T^+(1,2,1,1)$, we get the following dual graphs
\begin{center}
\includegraphics[width=8cm]{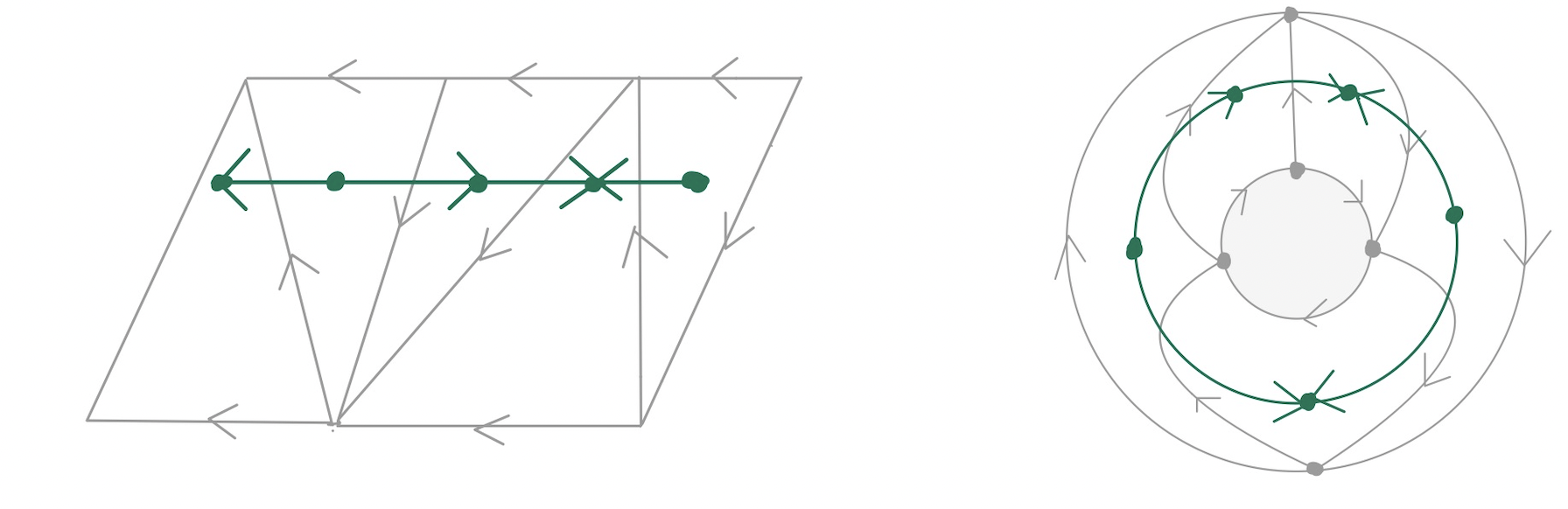}
\end{center}
The $q$-continuants have interpretations using the closures of the graph associated with the triangulated polygon, see \cite[\S3]{MGOfmsigma}. In particular Theorem 4 in \cite[\S3]{MGOfmsigma} implies a similar statement for the $q$-rotundus using the closures of the cyclic graph associated with the triangulated annulus:
$$
R^{+}(a_1,a_2,\ldots,a_{2m})_q=\sum_{C \in \G} q^{\# C},
$$
where the sum runs over all the closures $C$ of the cyclic dual graph $\G$ of $\T^{+}(a_1,a_2,\ldots,a_{2m})$.

This can be also formulated in terms of ranks of ideals in some posets (fence posets and circular posets), see \cite{McCSS}, \cite{OgRa}, \cite{Og}.
\subsection{Pfaffians}Recall that the determinant of a skew-symmetric matrix can always be written as the square of a polynomial expression in the entries of the skew-symmetric matrix. This polynomial expression is called the \textit{Pfaffian} of the matrix.
It is proved in \cite{CoOv2} that the rotundus  is the Pfaffian of a bigger skew-symmetric matrix. 
This can be generalized. 

The 
 $q$-rotundus is the Pfaffian of a $2k\times 2k$ skew-symmetric matrix:
\begin{equation*}
\label{TheOmEq}
\begin{array}{l}
R_k(c_1,\ldots,c_{k})_{q}^2=\\[6pt]
\qquad\det
\left(
\begin{array}{cccc|cccccc}
&&&\quad1&[c_1]_{q}&1&\\[4pt]
&&&&q^{c_{1}-1}&[c_2]_{q}&1&\\[4pt]
&&&&&q^{{c_{2}-1}}&\ddots&\ddots\\[4pt]
&&&&&&\ddots&\ddots&1\\[4pt]
-1&&&&&&&\ddots&[c_{k}]_{q}\\[6pt]\hline
&&&&&&&&\\[-8pt]
-[c_1]_{q}&-q^{c_{1}-1}&&&&&&&q^{c_{k}-1}\\[4pt]
-1&\ddots&\!\!\!\!\!\!\ddots&&&&&&\\[4pt]
&\ddots&\ddots&\\
&&&\;\;-q^{c_{k-1}-1}&&&&\\[6pt]
&&-1&-[c_k]_{q}&\!\!-q^{c_{k}-1}&&
\end{array}
\right)
\end{array}
\end{equation*}
This is a $q$-analogue of Theorem 1 in \cite{CoOv2} and this can be established using the same proof.

The $q$-rotundus also appears in the determinant of a $2k\times 2k$ symmetric matrix:
\begin{equation*}
\label{TheOmEq2}
\begin{array}{l}
R_k(c_1,\ldots,c_{k})_{q}^2-4q^{\sum_{i}(c_{i}-1)}=\\[10pt]
\qquad (-1)^k\det
\left(
\begin{array}{cccc|cccccc}
&&&\quad1&[c_1]_{q}&1&\\[4pt]
&&&&q^{c_{1}-1}&[c_2]_{q}&1&\\[4pt]
&&&&&q^{{c_{2}-1}}&\ddots&\ddots\\[4pt]
&&&&&&\ddots&\ddots&1\\[4pt]
1&&&&&&&\ddots&[c_{k}]_{q}\\[6pt]\hline
&&&&&&&&\\[-8pt]
[c_1]_{q}&q^{c_{1}-1}&&&&&&&q^{c_{k}-1}\\[4pt]
1&[c_2]_{q}&q^{{c_{2}-1}}&&&&&&\\[4pt]
&\ddots&\ddots&\\
&&&\!\!\!\!\ddots&&&&\\[6pt]
&&1&[c_k]_{q}&\!\!q^{c_{k}-1}&&
\end{array}
\right)
\end{array}
\end{equation*}
This identity has been checked experimentally with computer assistance for values of $k\leq 5$ and various tuples of  $c_i$'s. We conjecture that the formula holds in general. This would be a $q$-analogue of  the formula in the remark following Theorem 1 in \cite{CoOv2}. We also remark that the quantity $R_k(c_1,\ldots,c_{k})_{q}^2-4q^{\sum_{i}(c_{i}-1)}$ already appears in a formula of Proposition 4.3 of \cite{LMGadv}.

\subsection{Euler-Minding algorithm} The Euler-Minding formula gives the terms in the continuants by removing successively pairs $c_ic_{i+1}$ in the product $c_1c_2\cdots c_k$ see e.g \cite[p.9]{Per}. Conley-Ovsienko introduced a cyclic variant of this algorithm to compute the rotundus, see \cite[p46]{CoOv2}. 

We adapt these algorithms in the case of $q$-continuant and $q$-rotundus.

The $q$-continuant $E_k(c_1,\ldots,c_{k})_{q}$ can be calculated  as the sum of all terms obtained from the product $[c_{1}]_{q}[c_{2}]_{q}\cdots [c_{k}]_{q}$ by replacing all the adjacent pairs $[c_{i}]_{q}[c_{i+1}]_{q}$ by $-q^{c_{i}-1}$. It is possible to remove from 0 to $\lfloor k/2 \rfloor$ pairs at once.
For example,
$$
\begin{array}{rcl}
E_3(c_1,c_2,c_3)_{q}&=&[c_{1}]_{q}[c_{2}]_{q}[c_{3}]_{q}-q^{c_{1}-1}\cancel{[c_{1}]_{q}[c_{2}]_{q}}[c_{3}]_{q}-q^{c_{2}-1}[c_{1}]_{q}\cancel{[c_{2}]_{q}[c_{3}]_{q}}
\\[10pt]
&=&[c_{1}]_{q}[c_{2}]_{q}[c_{3}]_{q}-q^{c_{1}-1}[c_{3}]_{q}-q^{c_{2}-1}[c_{1}]_{q}
\\
\\[10pt]
E_4(c_1,c_2,c_3,c_{4})_{q}&=&
[c_{1}]_{q}[c_{2}]_{q}[c_{3}]_{q}[c_{4}]_{q}
-q^{c_{1}-1}\cancel{[c_{1}]_{q}[c_{2}]_{q}}[c_{3}]_{q}[c_{4}]_{q}
-q^{c_{2}-1}[c_{1}]_{q}\cancel{[c_{2}]_{q}[c_{3}]_{q}}[c_{4}]_{q}
\\[2pt]
&&\;\;-q^{c_{3}-1}{[c_{1}]_{q}}[c_{2}]_{q}\cancel{[c_{3}]_{q}[c_{4}]_{q}}
+
q^{c_{1}+c_{3}-2}\cancel{[c_{1}]_{q}[c_{2}]_{q}}\cancel{[c_{3}]_{q}[c_{4}]_{q}}
\\[10pt]
&=&
[c_{1}]_{q}[c_{2}]_{q}[c_{3}]_{q}[c_{4}]_{q}
-q^{c_{1}-1}[c_{3}]_{q}[c_{4}]_{q}
-q^{c_{2}-1}[c_{1}]_{q}[c_{4}]_{q}
-q^{c_{3}-1}{[c_{1}]_{q}}[c_{2}]_{q}\\[2pt]
&&\;\;
+q^{c_{1}+c_{3}-2}
\end{array}
$$
This algorithm can be deduced from the standard formula expressing the determinant of a $n\times n $-matrix: $\det(a_{i,j})=\sum_{\sigma \in S_n}\mathrm{sgn}(\sigma)a_{1,\sigma(1)}\ldots a_{n,\sigma(n)}$. In the case of the three-diagonal determinant \eqref{KEq} the formula reduces to the set of permutations $\sigma$ in the symmetric group $S_n$  that are product of elementary transpositions $(i, i+1)$ with disjoint supports. This explains the algorithm.

Using \eqref{qRotsEq} we derive a similar algorithm for the $q$-rotundus.
The rotundus $R_k(c_1,\ldots,c_{k})_{q}$ can be calculated  as the sum of all terms obtained from the product $[c_{1}]_{q}[c_{2}]_{q}\cdots [c_{k}]_{q}$ by replacing all the cyclically adjacent pairs $[c_{i}]_{q}[c_{i+1}]_{q}$ by $-q^{c_{i}-1}$. Here $[c_{k}]_{q}[c_{1}]_{q}$ is considered as an adjacent pair and is replaced by $-q^{c_{k}-1}$. It is possible to remove from 0 to $\lfloor k/2 \rfloor$ pairs at once.
For example,
$$
\begin{array}{rcl}
R_3(c_1,c_2,c_3)_{q}&=&[c_{1}]_{q}[c_{2}]_{q}[c_{3}]_{q}-q^{c_{1}-1}\cancel{[c_{1}]_{q}[c_{2}]_{q}}[c_{3}]_{q}-q^{c_{2}-1}[c_{1}]_{q}\cancel{[c_{2}]_{q}[c_{3}]_{q}}
\\[2pt]
&&\;\;-q^{c_{3}-1}\cancel{[c_{1}]_{q}}[c_{2}]_{q}\cancel{[c_{3}]_{q}}\\[10pt]
&=&[c_{1}]_{q}[c_{2}]_{q}[c_{3}]_{q}-q^{c_{1}-1}[c_{3}]_{q}-q^{c_{2}-1}[c_{1}]_{q}
-q^{c_{3}-1}[c_{2}]_{q}\\
\\[10pt]
R_4(c_1,c_2,c_3,c_{4})_{q}&=&
[c_{1}]_{q}[c_{2}]_{q}[c_{3}]_{q}[c_{4}]_{q}
-q^{c_{1}-1}\cancel{[c_{1}]_{q}[c_{2}]_{q}}[c_{3}]_{q}[c_{4}]_{q}
-q^{c_{2}-1}[c_{1}]_{q}\cancel{[c_{2}]_{q}[c_{3}]_{q}}[c_{4}]_{q}
\\[2pt]
&&\;\;-q^{c_{3}-1}{[c_{1}]_{q}}[c_{2}]_{q}\cancel{[c_{3}]_{q}[c_{4}]_{q}}
-q^{c_{4}-1}\cancel{[c_{1}]_{q}}[c_{2}]_{q}[c_{3}]_{q}\cancel{[c_{4}]_{q}}\\[2pt]
&&\;\;+
q^{c_{1}+c_{3}-2}\cancel{[c_{1}]_{q}[c_{2}]_{q}}\cancel{[c_{3}]_{q}[c_{4}]_{q}}
+q^{c_{2}+c_{4}-2}\cancel{[c_{1}]_{q}}\cancel{[c_{2}]_{q}[c_{3}]_{q}}\cancel{[c_{4}]_{q}}\\[10pt]
&=&
[c_{1}]_{q}[c_{2}]_{q}[c_{3}]_{q}[c_{4}]_{q}
-q^{c_{1}-1}[c_{3}]_{q}[c_{4}]_{q}
-q^{c_{2}-1}[c_{1}]_{q}[c_{4}]_{q}
-q^{c_{3}-1}{[c_{1}]_{q}}[c_{2}]_{q}\\[2pt]
&&\;\;
-q^{c_{4}-1}[c_{2}]_{q}[c_{3}]_{q}+q^{c_{1}+c_{3}-2}
+q^{c_{2}+c_{4}-2}
\end{array}
$$
Note that applying this algorithm to $R_k(c_k,\ldots,c_{1})_{q}$ would lead to other formulas that should simplify to the same polynomials in $q$.
For instance for $k=3$ one can check 
$$
[c_{1}]_{q}[c_{2}]_{q}[c_{3}]_{q}-q^{c_{1}-1}[c_{3}]_{q}-q^{c_{2}-1}[c_{1}]_{q}
-q^{c_{3}-1}[c_{2}]_{q}
=[c_{3}]_{q}[c_{2}]_{q}[c_{1}]_{q}-q^{c_{3}-1}[c_{1}]_{q}-q^{c_{2}-1}[c_{3}]_{q}
-q^{c_{1}-1}[c_{2}]_{q}.
$$
\bigskip

\noindent
\textbf{Acknowledgement.} 
SMG is very grateful to Patrick Popescu-Pampu for bringing to her knowledge interesting references, in particular \cite[Prop. 5.21]{PoP}.
The authors would like to thank Valentin Ovsienko and Christophe Reutenauer for stimulating discussions on the subject, Perrine Jouteur for pointing out major typos in the previous version of the article, and the anonymous referees
whose valuable comments helped to greatly improve the presentation of the paper.

\bibliographystyle{acm}
\bibliography{BiblioMoy3,qAnalog,BibFrisesCluster}\end{document}